\documentclass[reqno]{amsart}
\usepackage{amsthm}
\usepackage{amsmath}
\usepackage{amssymb}
\usepackage{graphicx}
\usepackage{xcolor}
\usepackage{hyperref}
\usepackage{todonotes}
\usepackage{comment}
\usepackage[a4paper, total={6in, 8in}]{geometry}

\theoremstyle{plain}
\newtheorem{thm}{\protect\theoremname}
\theoremstyle{remark}

\newtheorem{lemma}[thm]{Lemma}
\newtheorem{theorem}[thm]{Theorem}
\newtheorem{remark}[thm]{Remark}
\newtheorem{proposition}[thm]{Proposition}
\newtheorem{definition}[thm]{Definition}

\newcommand{\sign}{\text{sign}}

\begin{document}

\title{A Complete Spectral Analysis of the CEV Operator with Applications to Arbitrage}

\date{\today}

\begin{abstract}
We provide a complete Sturm–Liouville spectral analysis of the Constant Elasticity of Variance (CEV) operator. By transforming the corresponding Fokker–Planck operator into a generalized Laguerre operator, we explicitly characterize its self-adjoint extensions, boundary conditions, spectra, and eigenfunctions across all elasticity regimes. We then relate these spectral features to arbitrage phenomena in the CEV model, showing how boundary behavior and positive harmonic functions encode the distinction between attainable-boundary arbitrage mechanisms and strict-local-martingale bubble regimes. The result is an explicit operator-theoretic perspective on the link between CEV dynamics, no-arbitrage, and spectral theory. 

\end{abstract}

\author{Filippo Beretta}
\email{filippo.beretta@math.ethz.ch}
\address{Department of Mathematics, ETH Z\"{u}rich, Rämistrasse 101
8092 Z\"{u}rich, Switzerland}

\author{Florian Kogelbauer}
\email{floriank@ethz.ch}
\address{Department of Mathematics, ETH Z\"{u}rich, Rämistrasse 101
8092 Z\"{u}rich, Switzerland}

\maketitle
\let\thefootnote\relax\footnotetext{Filippo Beretta gratefully acknowledges partial support by the SNF project MINT 205121-21981.}

\section{Introduction}

Operator spectral theory provides a fundamental analytical framework in financial mathematics, offering a systematic way to study the generators and pricing operators associated with stochastic asset dynamics. In derivative pricing, for instance in the Black--Scholes framework, spectral decompositions of the infinitesimal generator lead to eigenfunction expansions and long-maturity asymptotics for option prices \cite{davydov2003pricing}. In interest-rate theory, particularly for affine term-structure models, spectral representations of pricing semigroups yield tractable formulas for bond and option valuation \cite{chazal2018option}. In long-run asset pricing, Perron--Frobenius methods applied to positive valuation operators characterize the long-run discount rate and risk adjustments through the dominant eigenvalue and eigenfunction \cite{hansen2009longterm}. In empirical term-structure analysis, spectral decompositions of covariance operators provide the basis for principal component analysis of yield curves and for the identification of level, slope, and curvature factors \cite{litterman1991common}.\\

In this work, we analyze the Constant Elasticity of Variance (CEV) model with elasticity parameter $\gamma$ through the lens of Sturm--Liouville theory. The model was first introduced by \cite{cox1996constant} for $\gamma \leq 2$ and was extended to the regime $\gamma>2$ in \cite{EmanuelMacBeth1982}. The CEV model has become a widely used tool in financial mathematics since it admits closed-form analytical pricing formulas for several classes of options \cite{Schroder1989,DavydovLinetsky2001,Delbean,CarrLinetsky2006}. We study the associated Fokker--Planck operator under general boundary conditions and discuss non-standard extensions from Hilbert spaces to Krein spaces. Our approach relies on transforming the CEV operator into an equivalent Laguerre operator, whose spectral theory is well understood \cite{Krall1979laguerre,derkach1998extensions}. Spectral properties of the CEV operator in mathematical finance have been analyzed in \cite{Fouque2018SpectralSV}. In \cite{Linetsky2004CEV}, it is shown that the generator exhibits a purely discrete spectrum in certain parameter regimes, while \cite{Linetsky2008Spectral} demonstrates that the associated pricing problem amounts to selecting a self-adjoint extension of the generator. Further connections are developed in \cite{CarrLinetsky2006}, where the CEV model with killing is interpreted in terms of Bessel processes, and in \cite{DavydovLinetsky2001}, where the eigenfunctions are derived and closed-form spectral pricing formulas are obtained.\\

Sturm--Liouville theory itself has several applications in mathematical finance, most prominently in the so-called Ross recovery \cite{ross2015recovery,von2020numerical}. In \cite{park2016ross}, the authors extend the continuous-time Ross recovery framework beyond recurrent state dynamics and show that, for transient diffusion, the risk-neutral measure alone is insufficient to identify the objective measure. Building on this perspective, \cite{ahn2020examining} establishes that Sturm--Liouville-based Ross recovery is feasible if and only if both boundary endpoints of the state space are limit-point under suitable integrability conditions, and further demonstrates how non-uniqueness emerges once these conditions are relaxed.\\

A central theme of the present paper is the connection between spectral theory and arbitrage. Arbitrage is one of the foundational concepts of mathematical finance. It underlies fair pricing, market equilibrium, and the absence of free lunches, linking probabilistic models with economic principles. It is typically understood as the possibility of earning a risk-free profit with zero net investment. Under an equivalent risk-neutral martingale measure, discounted traded-asset prices are martingales. This martingale property expresses no-arbitrage probabilistically, since discounted prices have zero drift after an equivalent change of measure. There is, however, also a functional-analytic viewpoint. From this perspective, no-arbitrage can instead be characterized by the existence of a strictly positive linear functional on the space of attainable payoffs, together with spectral restrictions on the associated pricing operator that rule out positive spectral growth \cite{jia2006caracterisation,kreps1981arbitrage}. Thus, the fundamental theorem of asset pricing has both a probabilistic formulation, in terms of martingale measures, and an operator-theoretic formulation, in terms of positive linear functionals and spectral properties.\\

This operator-theoretic viewpoint is especially natural when one moves from path-wise stochastic calculus to the forward-equation or Fokker--Planck formulation. Instead of following individual sample paths, one studies the evolution of probability densities under a linear parabolic differential operator. The infinitesimal generator of the underlying Markov process then becomes the central object: derivative prices are obtained by applying the associated semigroup to payoff functions, and long-maturity behavior is governed by the spectrum of this operator. Eigenvalues determine asymptotic discount rates, eigenfunctions encode principal modes of the pricing dynamics, and spectral gaps control convergence to equilibrium. In this way, spectral theory provides a unifying analytical framework linking no-arbitrage, martingale measures, and the dynamical evolution of prices.\\

\begin{figure}
    \centering
    \includegraphics[width=1\linewidth]{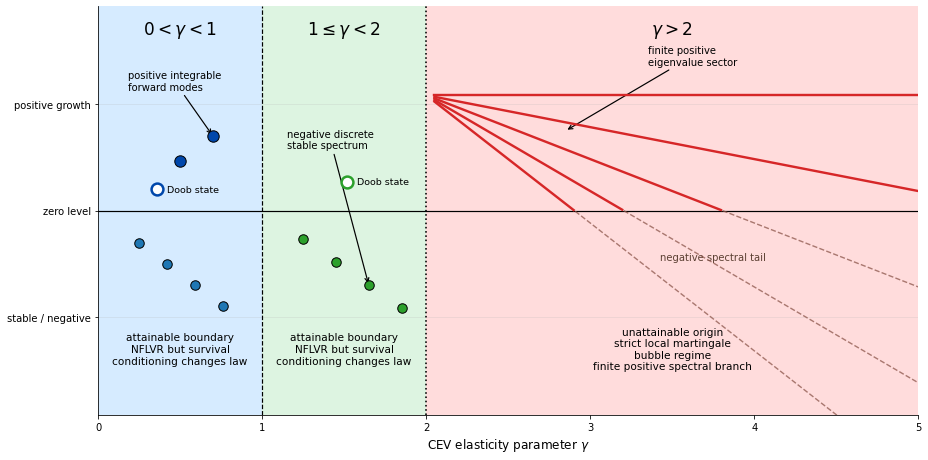}
    \caption{\footnotesize
Schematic spectral and arbitrage regimes of the CEV operator as the elasticity parameter $\gamma$ varies. For $0<\gamma<1$, the origin is attainable and limit-circle, and suitable self-adjoint extensions may admit positive integrable forward modes, which encode the boundary-conditioning mechanism and the associated Doob $h$-transform. For $1\leq\gamma<2$, the spectrum is negative and discrete, while survival conditioning is represented only by a generalized backward or boundary state. The value $\gamma=2$ is the singular Black--Scholes limit, requiring logarithmic rather than Laguerre coordinates. For $\gamma>2$, the origin is unattainable, so attainable-boundary arbitrage disappears. The relevant phenomenon is instead strict-local-martingale behavior, reflected schematically by a finite positive eigenvalue sector together with a negative spectral tail. The figure illustrates how boundary classification, positive spectral states, and arbitrage mechanisms are tied together in the CEV model.
}
    \label{CEV_Spectrum_plot}
\end{figure}

For the CEV model, the link between spectrum and arbitrage is particularly transparent because both are governed by the behavior of the diffusion at the boundary. When $\gamma<2$, the origin is attainable, and arbitrage phenomena are closely related to conditioning the process not to hit this absorbing boundary. This conditioning is naturally implemented by a Doob $h$-transform, where $h$ is a strictly positive harmonic function of the backward generator. At the operator level, the transform replaces the original generator $G_\gamma$ by
\[
    G_\gamma^h f := h^{-1}G_\gamma(hf),
\]
thereby producing a new positive semigroup and modifying the drift by a logarithmic derivative of $h$. In probabilistic terms, this corresponds to changing the law of the diffusion by conditioning on survival. In financial terms, it is precisely this change of law, which excludes paths reaching the absorbing boundary, that underlies the attainable-boundary arbitrage mechanism described for the CEV model through its connection with Bessel processes \cite{Delbean,CarrLinetsky2006}.\\

The spectral interpretation of this construction is that the relevant conditioning function is not an arbitrary auxiliary object, but a positive harmonic state of the pricing generator. Thus, the Doob transform selects a distinguished positive solution of the eigenvalue equation at eigenvalue zero, or equivalently a ground-state-type mode after a suitable spectral shift. The boundary classification determines whether this positive state is represented by an admissible forward eigenfunction or only by a generalized state of the adjoint problem. In the regime $0<\gamma<1$, suitable self-adjoint extensions admit positive integrable forward modes, so the conditioning mechanism is visible directly in the forward spectrum. In the regime $1\leq\gamma<2$, the forward spectrum is negative and discrete, and the same survival conditioning appears only through a generalized boundary state of the backward generator.\\

The regime $\gamma>2$ is different. In this case, the origin is unattainable, so conditioning the process to remain positive is trivial and no analogous boundary-induced Doob transform gives rise to the same arbitrage mechanism. Nevertheless, the discounted CEV price may fail to be a true martingale and become a strict local martingale, producing the bubble phenomenon first observed in \cite{EmanuelMacBeth1982}. Spectrally, this regime is reflected in a different boundary classification and in the presence of a finite positive eigenvalue sector together with a negative spectral tail. The distinction between attainable-boundary arbitrage for $\gamma<2$ and strict-local-martingale bubbles for $\gamma>2$ can therefore be read from the positive spectral and harmonic structure of the CEV operator, see Figure \ref{CEV_Spectrum_plot}.\\

The contribution of this paper is to make this correspondence explicit. We provide a complete Sturm--Liouville spectral analysis of the CEV Fokker--Planck operator across all elasticity regimes, characterize the relevant self-adjoint extensions and boundary conditions, and identify the spectral objects that encode arbitrage-related behavior. In particular, positive harmonic functions and positive eigenfunctions play a central role: they determine admissible Doob transforms, describe conditioning mechanisms at attainable boundaries, and distinguish these mechanisms from bubble regimes driven by strict local martingales.\\

The outline of the paper is as follows. In Section~\ref{sec:motivation}, we review the various notions of arbitrage for continuous-time stochastic processes and provide the intuition linking these properties to the operator spectrum. From that point forward, we specialize in the CEV model. Section~\ref{sec_spectrum_CEV} is devoted to a detailed, explicit analysis of the spectral properties of the CEV generator. Finally, Section~\ref{sec:arbitrage_CEV} focuses on the CEV model's arbitrage properties. We detail the connection between the operator's spectrum and no-arbitrage conditions, exploring the role of spectral operations such as spectral shifts and the application of Doob's $h$-transform.




\section{Motivation: Connections Between Arbitrage and Spectral Properties}\label{sec:motivation}
In this section, we first recall the fundamental definitions of arbitrage for continuous-time processes. For simplicity, we focus on scalar processes throughout. After elaborating the Fundamental Theorem of Asset Pricing and various notions of arbitrage, we move from a probabilistic approach to spectral properties of the generator. We recover classical results through an operator-theoretic framework thus motivating the detailed spectral analysis of the CEV process in Section \ref{sec_spectrum_CEV}. These results subsequently allow us to characterize arbitrage-related properties through a detailed spectral analysis as discussed in Section \ref{subsec:arb_spec}.


\subsection{The Fundamental Theorem of Asset Pricing}\label{subsec:FTAP}

We consider a market on a finite time horizon $[0,T]$ consisting of a risk-free bond $B$ and a risky asset $X$, defined on a filtered probability space $(\Omega,\mathcal{F},\mathbb{F},\mathbb{P})$. Let $W$ be a $\mathbb{P}$-Brownian motion and assume that $\mathbb{F}$ is the $\mathbb{P}$-augmented filtration generated by $W$, satisfying the usual conditions of right-continuity and completeness. The asset dynamics are given by
\begin{equation}\label{generic_market}
\begin{aligned}
\mathrm{d} B_t &= r B_t \,\mathrm{d}t, \quad B_0 = 1, \\
\mathrm{d} X_t &= \mu(X_t)\,\mathrm{d}t + \sigma(X_t)\,\mathrm{d}W_t, \quad X_0 = x \in \mathbb{R},
\end{aligned}
\end{equation}
where $r \ge 0$ is the risk-free rate. The coefficient functions $\mu: \mathbb{R} \to \mathbb{R}$ and $\sigma: \mathbb{R} \to \mathbb{R}$ are assumed to be regular enough to ensure existence and uniqueness of a weak solution to \eqref{generic_market}. An investment strategy is a pair $(\phi,\psi)$ of $\mathbb{F}$-predictable processes, representing holdings in the bond and the risky asset. The associated portfolio value is
\begin{equation}
V_t := \phi_t B_t + \psi_t X_t, \qquad t \in [0,T].
\end{equation}
We restrict attention to self-financing portfolios, meaning that changes in value arise only from asset price movements. Defining the discounted processes $U_t := e^{-rt}X_t$ and $\tilde V_t := e^{-rt}V_t$, the self-financing condition yields
\begin{equation}
\tilde V_t = V_0 + \int_0^t \psi_s \,\mathrm{d}U_s, \qquad t \in [0,T].
\end{equation}
Hence, a portfolio is fully characterized by the initial capital $V_0$ and the strategy $\psi$ and we henceforth denote it by $V^{V_0,\psi}$. We further restrict to portfolios bounded from below, thereby excluding doubling strategies. The set of admissible strategies, ensuring that $V^{0,\psi}$ is bounded from below, is denoted by $\mathcal{A}$. We denote by $\mathbb{L}^0_+(\mathbb{R},\mathcal{F}_T,\mathbb{P})$ and $\mathbb{L}^\infty(\mathbb{R},\mathcal{F}_T,\mathbb{P})$ the sets of non-negative and essentially bounded $\mathcal{F}_T$-measurable random variables, respectively. Then, we introduce the sets 
\begin{equation}
K_0(\mathbb{F},\mathbb{P}):= \big\{\tilde{V}^{0,\psi}_T, \;\psi \in \mathcal{A}\big\}, \quad C_0(\mathbb{F},\mathbb{P}) := K_0(\mathbb{F},\mathbb{P}) - \mathbb{L}^0_+(\mathbb{R},\mathcal{F}_T,\mathbb{P})
\end{equation}
to denote the set of portfolio values at time $T$ and of $\mathcal{F}_T$-measurable payoffs that can be dominated by the terminal value of an admissible portfolio. Finally, we introduce
\begin{equation}
C(\mathbb{F},\mathbb{P}) := C_0(\mathbb{F},\mathbb{P}) \cap \mathbb{L}^\infty(\mathbb{R},\mathcal{F}_T,\mathbb{P}).
\end{equation}

\begin{definition}
    The market satisfies the \emph{No Arbitrage} (NA) condition if
    \begin{equation}
    C(\mathbb{F,\mathbb{P}}) \cap \mathbb{L}^\infty_+(\mathbb{R},\mathcal{F}_T,\mathbb{P}) = \{0\}.
    \end{equation}
\end{definition}
The interpretation is as follows: an arbitrage opportunity is a strategy $\psi \in \mathcal{A}$ such that
\begin{equation}
\tilde V^{0,\psi}_T \ge 0 \quad \mathbb{P}\text{-a.s.}, \quad \mathbb{P}\big[\tilde V^{0,\psi}_T > 0\big] > 0.
\end{equation}
The (NA) condition requires that any such strategy satisfy $\tilde V^{0,\psi}_T = 0$ $\mathbb{P}$-almost surely. \\

In continuous time, a stronger and more appropriate notion is \emph{No Free Lunch with Vanishing Risk} (NFLVR), see \cite{delbaen1993nonarbitrage,delbaen1994arbitrage,delbaen2004whatis}.

\begin{definition}
    The market satisfies the \emph{No–Free Lunch with Vanishing Risk} (NFLVR) condition if
    \begin{equation}
    \overline{C}(\mathbb{F,\mathbb{P}}) \cap \mathbb{L}^\infty_+(\mathbb{R},\mathcal{F}_T,\mathbb{P}) = \{0\},
    \end{equation}
    where $\overline{C}(\mathbb{F},\mathbb{P})$ denotes the closure of $C(\mathbb{F,\mathbb{P}})$ in the norm of $\mathbb{L}^\infty(\mathbb{R},\mathcal{F}_T,\mathbb{P})$.
\end{definition}

It is clear that the (NFLVR) condition is stronger than (NA). To compare the two, we introduce the set $\mathcal{A}_1 \subset \mathcal{A}$ as the one of admissible strategies such that $\tilde{V}^{0,\psi}$ is bounded from below by $-1$. Similarly to $K_0(\mathbb{F},\mathbb{P})$, one may define 
\begin{equation}
K^1_0(\mathbb{F},\mathbb{P}):= \big\{\tilde{V}^{0,\psi}_T, \;\psi \in \mathcal{A}_1\big\}.
\end{equation}
The last notion of arbitrage we report is (NUPBR), introduced in \cite{karatzas2007numeraire}.
\begin{definition}
    The market satisfies the \emph{No Unbounded Profit Bounded Risk} (NUPBR) condition if $K_0^1(\mathbb{F},\mathbb{P})$ is bounded in $\mathbb{P}$-probability.
\end{definition}
The following lemma shows that (NUPBR) is the missing ingredient required to connect (NA) and (NFLVR).
\begin{lemma}
    The market satisfies (NFLVR) if and only if it satisfies both (NA) and (NUPBR).
\end{lemma}
We refer to \cite{karatzas2007numeraire,kardaras2010finitely} for the proof. As the main result of this section, we recall the Fundamental Theorem of Asset Pricing, see \cite{Schachermayer,delbaen1998fundamental,delbaen2006mathematics}.
\begin{theorem}[Fundamental Theorem of Asset Pricing]\label{FTAP}
Assume the market follows \eqref{generic_market} and that trading is restricted to self-financing portfolios with admissible strategies $\psi\in\mathcal{A}$. Then the market satisfies (NFLVR) if and only if there exists a probability measure $\mathbb{Q}$, equivalent to $\mathbb{P}$ on $\mathcal{F}_T$, such that the discounted price process $U$ is an $(\mathbb{F},\mathbb{Q})$-local martingale.
\end{theorem}

The measure $\mathbb{Q}$ is typically referred to as \emph{Equivalent Local Martingale Measure} (ELMM). Although this theorem provides a complete characterization of (NFLVR), it is generally difficult to verify in practice. As observed by \cite{karatzas2007numeraire}, there is generally no computationally feasible method to detect arbitrage or to verify the existence of an equivalent martingale measure directly from the asset dynamics.\\

For the single risky asset in \eqref{generic_market}, (NFLVR) admits a more explicit characterization (see \cite[Sections 1.4--1.5]{karatzas1998methods} for proofs). We begin presenting a necessary condition.
\begin{proposition}\label{prop_nec_condition}
    If (NFLVR) holds, there exists a $\mathbb{F}$-predictable process $\lambda$ satisfying 
    \begin{equation}\label{def_lambda}
    \sigma(X_t) \lambda_t=\mu(X_t)-r X_t, \quad \mathrm{d} t \otimes \mathrm{dP} \text {-a.e. on } [0, T] \times \Omega.
    \end{equation}
\end{proposition}
The (possibly not unique) process $\lambda$ is denoted as risk-premium process. We provide intuition on the connections of the risk-premium process to spectral properties in Section \ref{sect_spectral_arbitrale}. Assuming $\sigma$ is non-zero everywhere, the unique risk-premium process and its associated density process Z are defined as:
\begin{equation}\label{df_risk_premium_process}
\lambda_t := \frac{\mu(X_t) - rX_t}{\sigma(X_t)}, \quad Z_t := \exp\!\left(-\int_0^t \lambda_s \,\mathrm{d}W_s - \frac{1}{2}\int_0^t \lambda_s^2 \,\mathrm{d}s\right),\qquad t \in [0,T],
\end{equation}
If $\lambda$ is sufficiently integrable (e.g., it is square-integrable on $ [0,T] \times \Omega$), the process $Z$ is at least a $(\mathbb{F},\mathbb{P})$-local martingale. The following sufficient condition then links the martingality of $Z$ to the existence of an ELMM via Girsanov’s Theorem.
\begin{proposition}\label{prop_suff_cond}
    If the process $Z$ is a $(\mathbb{F},\mathbb{P})$-martingale, then the measure $\rm d \mathbb{Q} = Z_T \rm d \mathbb{P}$ defines an ELMM, and (NFLVR) holds.
\end{proposition}
Bessel processes provide classic examples of arbitrage arising from the absence of an ELMM \cite{DelbaenSchachermayer1995}, a property later used to construct arbitrage in the CEV model for $\gamma < 2$ \cite{Delbean}. As explored in \cite{Fontanaetal2014, RufRunggaldier2014}, markets failing (NFLVR) but satisfying (NUPBR) remain viable for pricing and hedging. Intuitively, this corresponds to the density process $Z$ being a strict local martingale rather than a true martingale. It should be emphasized that even when the state price density $Z$ is a $(\mathbb{F},\mathbb{P})$-martingale, the discounted asset price $U$ is only only guaranteed to be an $(\mathbb{F},\mathbb{P})$-local martingale. This distinction forms the mathematical foundation for the theory of asset price bubbles (see \cite{Cox,Heston,jarrow2007complete, jarrow2010incomplete}). Notably, in the CEV model, the discounted price becomes a strict local martingale when the volatility exponent satisfies $\gamma>2$. This was first noticed by \cite{EmanuelMacBeth1982}, directly checking that the process does not keep constant expectation over time. The results were later extended by \cite{MijatovicUrusov2012}.

\subsection{Connections of Arbitrage to  Spectral Properties}\label{sect_spectral_arbitrale}
In this section, we provide a heuristic link between arbitrage and spectral theory by deriving the risk-premium $\lambda$ as defined in \eqref{def_lambda} from an operator-theoretic perspective. This spectral approach is inspired by \cite{hansen2009long}, who analyze long-term pricing functionals by assuming the generator admits a strictly positive eigenfunction $x\mapsto \phi(x)>0$ with eigenvalue $\rho \in \mathbb{R}$. To this end, denote the infinitesimal generator associated with the process $X_t$ in \eqref{generic_market} by
\begin{equation}\label{generator}
    \mathcal{G}(f) := \mu(x)\frac{\partial f}{\partial x}+\frac{\sigma^2(x)}{2}\frac{\partial^2 f}{\partial x^2},
\end{equation}
acting on observables $f$. We define the process
\begin{equation}
M_t := e^{\rho t}\frac{\phi(X_t)}{\phi(X_0)}, \quad t\in [0,T]
\end{equation}
for a pair $(\rho,\phi)$ where $\phi$ is a strictly positive function and $\rho \in \mathbb{R}$\footnote{We changed the notation from \cite{hansen2009long} where $1/\phi$ is used in Corollary 6.1 instead of $\phi$. Since $\phi$ is assumed to be strictly positive, this distinction is purely notational. While we require $\phi>0$ to ensure a valid change of measure, the eigenvalue $\rho$ remains unrestricted in sign.}. This process serves simultaneously as a discount factor and a change-of-measure density. We will show in the following that under (NFLVR), the construction of $M_t$ admits a spectral interpretation and recovers the classical relation $
M_t = e^{-rt} Z_t$, $t\in [0,T]$.\\
We assume that (NFLVR) holds. This implies that the process $MB$ is a $(\mathbb{F},\mathbb{P})$-martingale and $MX$ a $(\mathbb{F},\mathbb{P})$-local martingale. For the former, the intuition is clear: the expected value of holding money over time must be constant and equal to $1$ up to discounting. Specifically, for
\begin{equation}
M_t B_t = e^{(\rho +r) t}\frac{\phi(X_t)}{\phi(X_0)}, \quad t\in [0,T],
\end{equation}
to satisfy $\mathbb{E}[M_t B_t]=1$, its drift must vanish almost surely. Applying Ito's formula and the generator $\mathcal{G}$ defined in \eqref{generator}, this condition can be written as 
\begin{equation}\label{eigenfunction_cond_1}
 \mathcal{G}\phi + (r+\rho)\phi=0,
\end{equation}
hence $\phi$ has to be an eigenfunction of $\mathcal{G}$. Ensuring that $MX $ is a $(\mathbb{F},\mathbb{P})$-local martingale is the counterpart of the existence of an ELMM. Since $M$ is a strictly positive density, vanishing drift for the process $M_t X_t = e^{\rho t}X_t \frac{\phi(X_t)}{\phi(X_0)}$ yields the condition
\begin{equation}\label{eigenfunction_cond_2}
\mathcal{G}(\phi x) + \rho \phi x =0.
\end{equation}
Together, \eqref{eigenfunction_cond_1} and \eqref{eigenfunction_cond_2} characterize no-arbitrage as a system of eigenvalue conditions for the operator $\mathcal{G}$. \\
Combining \eqref{eigenfunction_cond_1} and \eqref{eigenfunction_cond_2}, we obtain the following condition on $\phi$
$$
\frac{\mathcal{G}(\phi x)}{\phi x} - \frac{\mathcal{G}(x)}{\phi} = r.
$$
Using the definition of $\mathcal{G}$, this can be rewritten as
$$
\frac{\phi'}{\phi}(x) = \frac{\mathrm{d}}{\mathrm{d}x}\ln (\phi)(x) = \frac{rx - \mu(x)}{\sigma^2(x)}, \quad x\in \mathbb{R}.
$$
The eigenfunction $\phi$ and the risk-premium process $\lambda$ defined in \eqref{df_risk_premium_process} are thus related through the following formula:
\begin{equation}
 \frac{\mathrm{d}}{\mathrm{d}x}\ln (\phi)(x)=  - \frac{\lambda(x)}{\sigma(x)}, \quad x\in \mathbb{R}.
\end{equation}
We have thus recovered the existence of $\lambda$ through a purely spectral lens and provide an operator-theoretic basis for the necessary condition of (NFLVR) established in Proposition \ref{prop_nec_condition}.\\
We now complete the description of the process $M$ connecting it to the density process $Z$. Making use of \eqref{eigenfunction_cond_1} and \eqref{eigenfunction_cond_2}, Ito's formula yields
\begin{equation}
\mathrm{d}\ln(\phi(X_t)) = - \bigg[({\rho + e}) + \frac{1}{2}\lambda_t^2 \bigg]\mathrm{d}t - \lambda_t \mathrm{d}W_t, \quad t\in [0,T].
\end{equation}
Notably, recalling the definition of the process $Z$ in \eqref{df_risk_premium_process}, 
\begin{equation}
\phi(X_t) = \phi(X_0) e^{-(\rho + r)t}Z_t, \quad t\in [0,T].
\end{equation}
We have therefore established that the process $M$ fulfilling the roles of discounting and density simultaneously is $M=Z/B$. This result confirms that the spectral arguments of this section are consistent with the standard arbitrage theory.

\section{A Spectral Characterization of the CEV operator}\label{sec_spectrum_CEV}

The CEV model is a classical example of a local volatility model that incorporates stochastic volatility together with the leverage effect \cite{cox1996constant}. The risky asset $X$ solves the one-dimensional SDE
\begin{equation}\label{CEV}
    d X_t=\mu X_t \mathrm{d}t+\sigma X_t^{\frac{\gamma}{2}} \mathrm{d} W_t, \quad t \in[0, T], \quad X_0=x\geq0,
\end{equation}
where $\mu,\gamma, \sigma \in (0,+\infty)$ are the drift, elasticity and volatility parameters, while $x$ is the initial price. \\
We distinguish two fundamental parameter regimes of \eqref{CEV}, the \emph{bubble-free regime}  $\gamma<2$, characterized by sub-linear market elasticity, and the \emph{bubble regime} $\gamma>2$, which exhibits super-linear elasticity. At the boundaries, the model recovers the Black--Scholes model ($\gamma=2$), the affine Cox--Ingersoll--Ross model ($\gamma=1$) \cite{cox1985theory}, and arithmetic Brownian motion, i.e., the one-dimensional Ornstein--Uhlenbeck process ($\gamma=0$) \cite{risken1996fokker}. The parameter range $\gamma<1$ allows for the volatility of a stock to increases as its price falls and the leverage ratio increases accordingly, which is a feature commonly observed in equity markets \cite{yu2005leverage}. On the other hand, fo rthe parameter range $\gamma>1$, the volatility of the price tends to increase as its price increases and its leverage ratio decreases, which is typically observed in commodity markets \cite{geman2009modeling}. The case of $\gamma <0$ is excluded as unrealistic since volatility would vanish as prices rise, while assets close to zero tend to exhibit high relative volatility and possibly increased risk of default or structural regime changes \cite{fouque2000derivatives}. \\
The Fokker-Planck equation \cite{risken1996fokker} associated to system \eqref{CEV} is given by
\begin{equation}\label{FP_eq_CEV}
\frac{\partial p_X}{\partial t} (t, x)=-\frac{\partial}{\partial x}[\mu x p_X(t, x)]+\frac{\partial^2}{\partial x^2}\left[\frac{\sigma^2 x^{\gamma}}{2} p_X(t, x)\right],
\end{equation}
for the probability density function $p_X$ of the process $X$. The operator
\begin{equation}\label{defLx}
    \mathcal{L}_{\gamma}[p] := -\mu\frac{\partial}{\partial x}[ x p]+\frac{\partial^2}{\partial x^2}\left[\frac{\sigma^2 x^{\gamma}}{2} p\right]. 
\end{equation}
is called \emph{Fokker--Planck operator} (FP operator), while its Hermitian conjugate,
\begin{equation}
    \mathcal{G}_\gamma[f] = \mu x \frac{\partial f}{\partial x}+\frac{\sigma^2 }{2}x^{\gamma}\frac{\partial^2 f}{\partial x^2},
\end{equation}
with $\mathcal{G}_\gamma =  \mathcal{L}_\gamma^*$, is called \emph{generator}\footnote{In the pricing literature, the generator is often denoted as $\mathcal{L}$, while the FP operator is then given as $\mathcal{L}^*$. We follow the notation common in statistical physics, where the FP operator is denoted as $\mathcal{L}$ and the generator is its conjugate.}. The FP operator acts on probability densities, while the generator acts on observables. For a general discussion of one-dimensional diffusion processes, including details on the relation of the FP operator and the generator, we refer to \cite{feller1954diffusion}. The stationary density of $\mathcal{L}_\gamma$, corresponding to an eigenfunction with eigenvalue zero, is given explicitly by
\begin{equation}
    \pi(x) = x^{\gamma} e^{\nu x^{2-\gamma}},\quad x>0. 
\end{equation}
The multiplication transform $M_\pi[f](x) = \pi(x) f(x)$, mapping an observable to a density, relates the FP operator and the generator,
\begin{equation}
    \mathcal{G}_\gamma = M_\pi^{-1}\mathcal{L}_\gamma M_\pi. 
\end{equation}
The generator is symmetric on the space $L^2((0,\infty), \pi)$. 

\begin{remark}
The generator $\mathcal{L}_\gamma^*$ characterizes the infinitesimal evolution of expectations of sufficiently regular functionals of the underlying stochastic process and uniquely determines its associated Markov semigroup. More precisely, if $X$ denotes the process with generator $\mathcal{L}_\gamma^*$, the corresponding semigroup $P = (P_t)_{t\in [0,T]}$, defined by $P_t f(x)=\mathbb{E}_x[f(X_t)]$, satisfies the Kolmogorov backward equation $\partial_t u=\mathcal{L}_\gamma^* u$ with initial condition $u(0,\cdot)=f$. Conversely, under suitable domain and growth conditions, the semigroup can be recovered from the generator via the exponential formula $P_t=\exp(t\mathcal{L}_\gamma^*)$, establishing a one-to-one correspondence between the stochastic dynamics, the generator, and the associated transition probabilities.
\end{remark}
The second-order differential operator \eqref{defLx} is an example of a Sturm--Liouville  (SL) operator \cite{zettl2005sturm,teschl2012ordinary}, whose spectral theory is well-developed. Indeed, SL theory provides a natural setting for the eigenfunction analysis of \eqref{defLx} by defining a weighted Hilbert space on which $\mathcal{L}_\gamma$ is self-adjoint. General results concerning SL theory on weighted spaces are summarized in Appendix \ref{app:generic_SL}. Additionally, we denote the probability flux associated with the Fokker--Planck equation \eqref{defLx} as $F$.\\
Written out explicitly, the CEV operator \eqref{defLx} reads
\begin{equation}\label{Lexpl}
\mathcal{L}_{\gamma}[p] =  \left[\frac{\sigma^2}{2}\gamma(\gamma-1)x^{\gamma -2}-\mu\right] p+\left[ \sigma^2 \gamma xx^{\gamma -2}-\mu x\right]\frac{\partial p}{\partial x}(x) + \frac{\sigma^2}{2}x^\gamma \frac{\partial ^2 p}{\partial x^2}, 
\end{equation}
which will be the basis for our further analysis. 
\begin{remark}
For $\gamma=2$, the Fokker--Planck operator takes the form 
\begin{equation}\label{L2}
    \mathcal{L}_{2}[p] = -\mu\frac{\partial}{\partial x}[ x p]+\frac{\partial^2}{\partial x^2}\left[\frac{\sigma^2 x^2}{2} p\right],
\end{equation}
the well-known Black--Scholes model. Log-normal coordinates 
\begin{equation}
    y = \log x,
\end{equation}
transform equation \eqref{L2} to the following second-order differential operator with constant coefficients, 
\begin{equation}
\tilde{\mathcal{L}}_2[p] = (\sigma^2-\mu)p+\left( \frac{3}{2}\sigma^2-\mu\right)\frac{\partial p}{\partial y }+ \frac{\sigma^2}{2}\frac{\partial^2 p}{\partial y^2},
\end{equation}
which can be readily integrated. 
\end{remark}
To ease notation, we define the following set of parameters:
\begin{equation}
    \begin{split}
        \beta & =2-\gamma, \quad 
\alpha  =\frac{3-\gamma}{2 \mu} \sigma^2 \\
\eta & =\frac{1}{2} \beta^2 \sigma^2, \quad 
\nu =-\frac{\beta\mu}{\eta}.
    \end{split}
\end{equation}
It is immediate to verify that
\begin{equation}
    \nu = \frac{2\mu}{(\gamma-2)\sigma^2},
\end{equation}
and consequently, 
\begin{equation}
    \begin{split}
          \gamma <2 \iff \nu < 0,\\
          \gamma >2 \iff \nu > 0.
    \end{split}
\end{equation}
For the parameter range $\gamma \in (2,+\infty)$, we assume the boundary condition
\begin{equation}\label{bc_gamma>2}
\lim_{x\to +\infty}  -\frac{p(x)x^{\gamma-1}}{\gamma-2 } + \theta \frac{\nu^{\frac{1}{2-\gamma}} }{\eta} x\left[F(p(x))+\frac{\sigma^2}{2}p(x)x^{\gamma-1}\right]=0
\end{equation}
where $\theta \in \overline{\mathbb{R}}$. The case $\theta =\infty$ is of particular interest as it allows for an explicit eigenfunction representation via generalized Laguerre polynomials. Under this choice, \eqref{bc_gamma>2} reduces to
\begin{equation}
\lim_{x\to +\infty}   x\left[F(p(x))+\frac{\sigma^2}{2}p(x)x^{\gamma-1}\right]=0
\end{equation}
Similarly, $\gamma \in [0,1)$, we impose the condition
\begin{equation}\label{bc_gamma_1_2}
    \lim_{x\to 0}  \frac{\gamma}{\gamma-2 }p(x)x^{\gamma-1} + \frac{x^{\gamma}}{2-\gamma}\frac{\partial p}{\partial x}(x) - \frac{\theta |\nu|^{\frac{1}{2-\gamma}}}{(2-\gamma)^2}x^{\gamma} \left[x\frac{\partial p}{\partial x}(x)+ (\gamma-1)p(x)\right]=0.
\end{equation}
For $\theta = \infty$, this simplifies to
\begin{equation}
\lim_{x\to 0} x^\gamma \left[x p'(x) + (\gamma-1)p(x)\right] = 0.
\end{equation}
The necessity and formal derivation of these conditions follow from the spectral analysis of the Laguerre operator; see Appendices \ref{sec:spectrum_Laguerre} and \ref{app:boundary_conditions} for details.

\subsection{The Normal Form of the CEV Operator}

Since prices can only be positive, we consider the operator \eqref{defLx} on the positive half-line, i.e., $I=(0,\infty)$. Let us first rewrite the Fokker--Planck operator for the CEV model in Sturm--Liouville normal form. Applying Lemma \eqref{LemmaNF} to the expanded form of the CEV operator in \eqref{Lexpl} we find that
\begin{equation}
    \begin{split}
        l_0(x) & = \frac{\sigma^2}{2}\gamma(\gamma-1)x^{\gamma -2}-\mu,\\
        l_1(x) & =  \sigma^2 \gamma x^{\gamma -1}-\mu x,\\
    l_2(x) & = \frac{\sigma^2}{2}x^\gamma ,
    \end{split}
\end{equation}
and consequently 
\begin{equation}
    \begin{split}
        \int_{x_0}^x \frac{l_1(\xi)}{l_2(\xi)}\, d\xi & = \int_{x_0}^x \frac{\sigma^2\gamma \xi|\xi|^{\gamma-2}-\mu \xi}{\frac{\sigma^2}{2}|\xi|^\gamma}\, d\xi\\
        & = 2\gamma \log x - \frac{2\mu}{\sigma^2(2-\gamma)} x^{2-\gamma} + C,
    \end{split}
\end{equation}
whenever $\gamma\neq 2$, for some $C\in\mathbb{R}$, which gives the weight function 
\begin{equation}\label{eq_wei_x}
    \omega_\gamma(x) =  x^\gamma \exp\left( \nu x^{2-\gamma} \right),
\end{equation}
and the functions
\begin{equation}\label{Q0Q1CEV}
\begin{split}
    Q_0(x) & = \omega_\gamma(x) \left[\frac{\sigma^2}{2}\gamma(\gamma-1)x^{\gamma -2}-\mu\right],\\
    Q_2(x) & = \frac{\sigma^2}{2}x^{\gamma} \omega_\gamma(x). 
    \end{split}
\end{equation}

\begin{figure}
    \centering
    \includegraphics[width=0.45\linewidth]{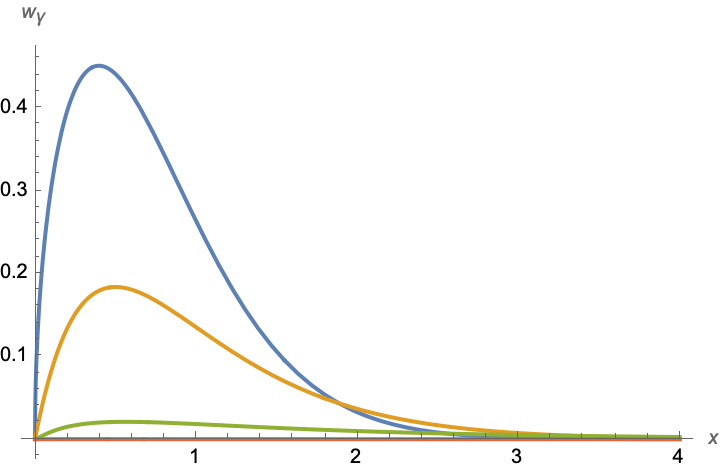}
\includegraphics[width=0.45\linewidth]{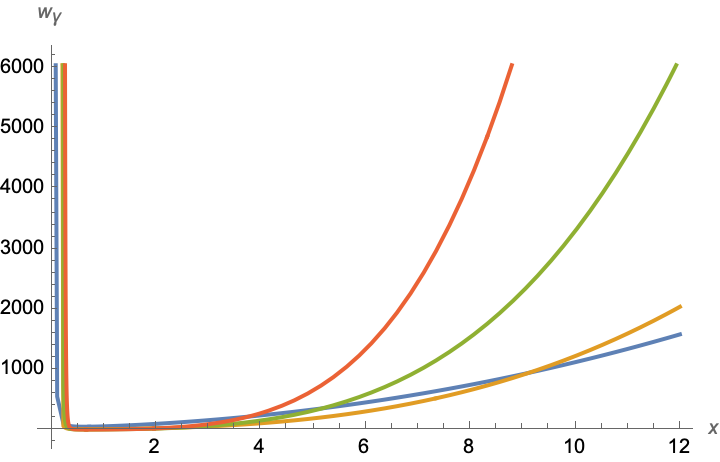}
    \caption{Weight function $w_\gamma$ for different values of $\gamma$: $\gamma = 0.5$ (left,blue), $\gamma = 1$ (left, orange), $\gamma = 1.5$ (left, green), $\gamma = 1.8$ (left, red), $\gamma = 2.5 $ (right, blue), $\gamma = 3$ (right, orange), $\gamma = 3.5 $ (right, green), $\gamma = 4$ (right, red). }
    \label{weightplot}
\end{figure}

\begin{remark}\label{remk:bc_0}
We emphasize that the the weight function $w_\gamma$ vanishes at the boundary $x=0$ for $\gamma<2$, while it becomes unbounded for $\gamma>2$, see Figure \ref{weightplot} for plots of $w_\gamma$ for varying elasticities. In particular, for $\gamma>2$, the weight function has an essential singularity at the origin and any element $p\in L^2(0,\infty,w_\gamma)$ necessarily satisfies $p(0)=0$. This is consistent with the observation that for any $t\in [0,T]$ and $x>0$, the state $X_t =0$ is never attainable for $\gamma >2$, up to null sets. This statement builds upon the connection of the CEV model and Bessel processes, as done in [\cite{Delbean}, Remark 3.7]. 
\end{remark}

\begin{remark}
For $\gamma = 2$ (Black--Scholes model), we have that
\begin{equation}\label{w2}
    w_2(x) = x^q,\quad q= \frac{2(\sigma^2-\mu)}{\sigma^2},
\end{equation}
and consequently
 \begin{equation}
     Q_2(x) = \frac{\sigma^2}{2}x^{\frac{4\sigma^2-2\mu}{\sigma^2}},
 \end{equation}
 which is zero at the origin for $2\sigma^2>\mu$, implying that the origin is l.p. in that case. The properties of the Black--Scholes model as a SL operator thus depend on the relative magnitude of $\sigma$ and $\mu$. 
\end{remark}


\subsection{Transformation of the CEV Fokker--Planck Operator to the Generalized Laguerre Operator}

By Ito's Lemma, the change of coordinates $f:\mathbb{R}^+\to\mathbb{R}^+$, $f(X) = |\nu| X^{2-\gamma}$, 
\begin{equation}\label{trans}
    Y_t = |\nu|X_t^{2-\gamma}, 
\end{equation}
transforms the SDE \eqref{CEV} into 
\begin{equation}\label{eqy}
     d Y_t=(2-\gamma)\Big[\mu Y_t +\frac{|\nu|}{2}(1-\gamma)\sigma^2\Big] \mathrm{d}t+(2-\gamma)\sigma \sqrt{|\nu| Y_t} \mathrm{d} W_t, \quad t \in[0, T], \quad Y_0=|\nu|x_0^{2-\gamma}\geq 0
\end{equation}
The Fokker--Planck equation associated to \eqref{eqy} reads
\begin{equation}\label{FP_eq_CEV_mod}
    \frac{\partial p_Y}{\partial t}=   \bar{\mathcal{L}}_\gamma[p_Y] , 
\end{equation}
for the operator
\begin{equation}\label{Lbar}
    \bar{\mathcal{L}}_\gamma[p_Y] = \eta |\nu|  \left[ y \frac{\partial^2 p_Y}{\partial y^2} +  \left(1+ \frac{1}{2-\gamma} + \sign(\nu)y\right)\frac{\partial p_Y}{\partial y}+  \sign(\nu)p_Y\right]. 
\end{equation}
The densities transform according to
\begin{equation}\label{pxpy}
    p_X(x)= |\nu|^\frac{1}{2-\gamma}(2-\gamma) y^{\frac{1-\gamma}{2-\gamma}} p_Y(y),
\end{equation}
This follows from transformation of the integration element,
\begin{equation}\label{dy_trans_dx}
    dy = |\nu|(2-\gamma)x^{1-\gamma}dx = |\nu|(2-\gamma) \bigg(\frac{y}{|\nu|}\bigg)^\frac{1-\gamma}{2-\gamma}dx.
\end{equation}
Detailed calculations showing how the density transformation \eqref{pxpy} induces the FP equation in \eqref{FP_eq_CEV_mod} are provided in Appendix \ref{checktransform}. Given \eqref{eq_wei_x}, the weight $w^Y_\gamma$ takes the following explicit form (up to a multiplicative factor):
\begin{equation}
    w^Y_\gamma(y) \sim  y^\frac{1}{2-\gamma} \exp\left( \sign(\nu) y \right).
\end{equation}
At the operator level, the coordinate change \eqref{trans} is realized via the linear mapping:
\begin{equation}
    U[p](x) = A\,x^{B}p(Cx^{D}), \qquad A\neq0,\; B\in\mathbb{R},\, C>0,\; D\in \mathbb{R}/\{0\}. 
\end{equation}
The inverse is given by 
\begin{equation}\label{Uinv}
U^{-1}[q](y)=\frac{1}{A}\left(\frac{y}{C}\right)^{-B/D}
\,q\!\left(\left(\frac{y}{C}\right)^{1/D}\right).
\end{equation}
To show that $U$ is a unitary operator from $L^2((0,\infty),w^Y_\gamma)$ to $L^2((0,\infty),w_\gamma)$, we verify it is an isometry: for any \(p\in L^2((0,\infty),w^Y_\gamma)\),
\begin{align*}
\|Up\|_{L^2((0,+\infty),w_\gamma)}^2
&=\int_0^\infty |Up(x)|^2  w_\gamma(x)\,dx \\
&=\int_0^\infty |A|^2 x^{2B}|p(Cx^D)|^2
\frac{|D|C}{|A|^2}x^{D-2B-1}w^Y_\gamma(Cx^D)\,dx \\
&=\int_0^\infty |D|C\,x^{D-1}|p(Cx^D)|^2 w^Y_\gamma(Cx^D)\,dx .
\end{align*}
With the substitution \(y=Cx^D\), we have \(dy=DC\,x^{D-1}dx\). Hence,
\begin{equation}
\|Up\|_{L^2((0,+\infty),w_\gamma)}^2
=\int_0^\infty |p(y)|^2 w(y)\,dy
=\|p\|_{L^2((0,+\infty),w^Y_\gamma)}^2.
\end{equation}
Notice that for $D<0$ the change of variables results in a flip of the boundaries, which is accounted for by the term $|D|$ in the weight. Thus \(U\) is an isometry and, because of \eqref{Uinv}, a unitary operator. Choosing the parameters: 
\begin{equation}
    A = |\nu|(2-\gamma),\quad B = 1-\gamma, \quad C = |\nu|, \quad D = (2-\gamma).
\end{equation}
we recover 
\begin{equation}
    p_X(x) =  U[p_Y](x) = |\nu|(2-\gamma) x^{1-\gamma} p_Y(|\nu|x^{2-\gamma})= |\nu|^\frac{1}{2-\gamma}(2-\gamma) y^{\frac{1-\gamma}{2-\gamma}} p_Y(y),
\end{equation}
consistent with the Ito-coordinate change \eqref{trans} and the transformation law \eqref{pxpy}. \\

Since the operator $U$ is unitary, the operator 
\begin{equation}\label{op_y}
   \bar{\mathcal{L}}_\gamma [p_Y]= U^{-1}\mathcal{L}_\gamma U[p_Y],
\end{equation}
corresponding to the right-hand side of \eqref{eqy} is iso-spectral to $\mathcal{L}_\gamma$, i.e., $\sigma(\mathcal{L}_\gamma) = \sigma(\bar{\mathcal{L}}_\gamma)$. 
For the case $\gamma >2$ ($\nu >0$), we apply yet another transform, a simple multiplication,
\begin{equation}\label{defM}
    p_Y(y) = M[g](y) = e^{-y}g(y),
\end{equation}
which is well-defined as an operator $M: L^2(0,\infty,e^{-2y}w^Y_\gamma)\to L^2(0,\infty,w^Y_\gamma)$. Note that \eqref{defM} is not unitary, but invertible, thus guaranteeing that $M^{-1}\mathcal{L}_\gamma M$ exists and that it is iso-spectral to $\mathcal{L}_\gamma$. An easy calculation shows that $\bar{\mathcal{L}}_\gamma$ is transformed to 
\begin{equation}\label{L_bar_M}
   (M^{-1}\bar{\mathcal{L}}_\gamma M)[g] = (\gamma-2)\mu \left[ y \frac{\partial^2 g}{\partial y^2} +  \left(1 + \frac{1}{2-\gamma} -y \right)\frac{\partial g}{\partial y} + \frac{1}{\gamma-2}g\right],
\end{equation}
The spectral properties of the original $\mathcal{L}_\gamma$-operator are therefore equivalent to the spectral properties of $\bar{\mathcal{L}}_\gamma$ and  $(M^{-1}\bar{\mathcal{L}}_\gamma M)$, respectively. The new weight behaves asymptotically as
\begin{equation}  e^{-2y}w_\gamma^Y(y)\sim  y^\frac{1}{2-\gamma} \exp\left( -y \right), \quad \gamma >2.
\end{equation}
We conclude this section summarizing the mapping of the spectral components. The eigenfunctions of the operator $\mathcal{L}_\gamma$ are related to those of the auxiliary operators \eqref{Lbar} and \eqref{L_bar_M} via the following transformations: 
\begin{align}
    p_X(x) &= -\nu(2-\gamma) x^{1-\gamma} p_Y(-\nu x^{2-\gamma}), \quad \gamma <2; \label{eq:pxpy_gamma<2}\\
    p_X(x) &= \nu(2-\gamma) x^{1-\gamma}e^{-\nu x^{2-\gamma}} g(\nu x^{2-\gamma}), \quad \gamma >2, \label{eq:pxpy_gamma>2}
\end{align}
where $p_Y$ and $g$ denote the eigenfunctions of the operators, respectively. Furthermore, it is immediate to see that \eqref{Lbar} for $\gamma <2$ and \eqref{L_bar_M} for $\gamma>2$ correspond precisely to the generalized Laguerre operator, defined in \eqref{defLa} and whose spectral theory is discussed in details in Appendix \ref{sec:spectrum_Laguerre}, once we set \begin{equation}\label{connection_a_gamma}
    a = \frac{1}{2-\gamma}.
\end{equation}
Hence, the spectrum of $\mathcal{L}_\gamma$ is determined by the spectrum of the Laguerre operator $\mathbf{L}_a$ as follows:
\begin{align}
    \sigma(\mathcal{L}_\gamma) &= \{\mu(2-\gamma)(\Lambda -1)\}_{\Lambda \in \sigma\left(\mathbf{L}_{\frac{1}{2-\gamma}}\right)}, \quad \gamma <2; \label{eq:sigma_gamma<2}\\
     \sigma(\mathcal{L}_\gamma) &= \left\{\mu(\gamma-2)\left(\Lambda +\frac{1}{\gamma-2}\right)\right\}_{\Lambda \in \sigma\left(\mathbf{L}_{\frac{1}{2-\gamma}}\right)}, \quad \gamma >2,\label{eq:sigma_gamma>2}
\end{align}

\begin{figure}
    \centering
    \includegraphics[width=0.8\linewidth]{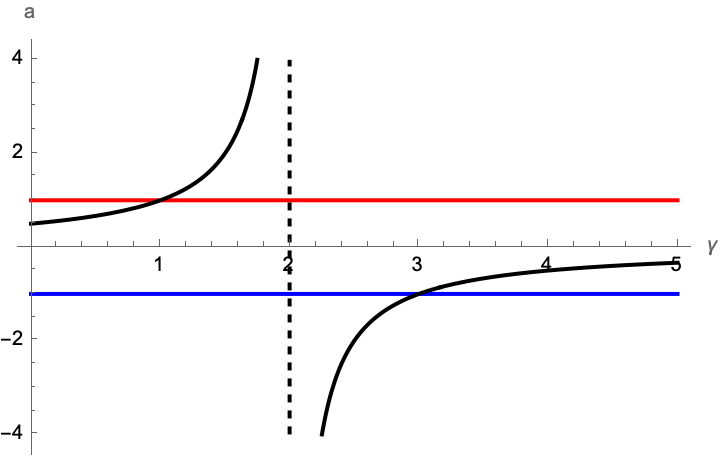}
    \caption{Behavior of the parameter $a = \frac{1}{2-\gamma}$ for $\gamma >0$. Red ($a=1$) and blue ($a=-1$) horizontal lines indicate the critical thresholds for the spectrum of the Laguerre operator $\mathbf{L}_{a}$ (see Theorem \ref{thm:spectru_Laguerre}).}
    \label{fig_pa}
\end{figure}

\subsection{Eigenfunctions and Boundary Conditions of the CEV Operator} In view of Theorem \ref{thm:spectru_Laguerre}, together with equation \eqref{eq:pxpy_gamma<2}, \eqref{eq:pxpy_gamma>2}, \eqref{eq:sigma_gamma<2} and \eqref{eq:sigma_gamma>2}, the spectrum and boundary conditions of the CEV operator can be readily determined. The detailed calculations of the transformation from the boundary conditions for \eqref{defLa} to the ones for \eqref{Lexpl} are deferred to Appendix \ref{app:boundary_conditions}. We distinguish various cases.\\

\noindent
\underline{$0\le\gamma < 1$ ($1/2\le a<1$)}:\\ At the boundaries, $\infty$ is a limit point, while $0$ is a limit circle. The boundary conditions necessary for a self-adjoint extension are characterized by \eqref{bc_gamma_1_2}. For the case where $\theta \in \mathbb{R}$, the spectrum is determined by the zeros of \eqref{eq:spec_theta_R} in conjunction with \eqref{eq:sigma_gamma<2}. For the self-adjoint extension $\theta = \infty$, the eigenvalues and eigenfunctions are provided explicitly in \eqref{eq:sp_theta_infty} and \eqref{eq:eigen_theta_infty}, respectively.\\

\noindent
\underline{$1\le\gamma < 2$ ($1\le a<\infty$)}:\\ Both $\infty$ and $0$ are limit point, and the CEV operator $\mathcal{L}_\gamma$ is self-adjoint on $L^2((0,\infty),w_\gamma)$. Its spectrum is discrete and is purely discrete, stable and given explicitly as 
\begin{equation}\label{eq:sp_theta_infty}
        \sigma(\mathcal{L}_\gamma) = \{-\mu(2-\gamma)(n+1)\}_{n\in \mathbb{N}}.
    \end{equation}
The corresponding eigenfunctions $\{p_n\}_{n\in \mathbb{N}} \subset L^2((0,\infty),w_\gamma)$ are given by the generalized Laguerre polynomials, pulled back to the $x$-variables, 
\begin{equation}\label{eq:eigen_theta_infty}
    p_n(x)  = x^{1-\gamma}L_n^\frac{1}{2-\gamma}(-\nu x^{2-\gamma}), \quad n\in\mathbb{N}.
\end{equation}

\medskip
\noindent
\underline{$\gamma> 2$ ($a<0$)}:\\
 The boundary at $0$ is a limit point, while the boundary at $\infty$ is a limit circle. The self-adjoint extension is characterized by the boundary condition \eqref{bc_gamma>2}. In the case where $\theta \in \mathbb{R}$, the spectrum is determined by the zeros of \eqref{eq:spec_theta_R} in conjunction with \eqref{eq:sigma_gamma>2}. For the self-adjoint extension $\theta = \infty$, the spectrum is given explicitly as 
 \begin{equation}\label{specLgam_ge_2}
    \sigma(\mathcal{L}_\gamma) =\left\{\mu-n\mu(\gamma-2)\right\}_{n\in \mathbb{N}}.
\end{equation}
Notably, the spectrum is purely discrete and consists of a finite number of positive eigenvalues. These correspond to each $n\in\mathbb{N}$ satisfying the condition $n \leq \frac{1}{\gamma-2}$. The associated eigenfunctions are readily computed and take the following form
\begin{equation}\label{eigenfunc_ge_2}
    p_n(x) =  x^{1-\gamma} e^{-\nu x^{2-\gamma}}L_n^\frac{1}{2-\gamma}(\nu x^{2-\gamma}), \quad n \in \mathbb{N}.
\end{equation}
We confirm  that $p_n(0)=0$ for any $n \in \mathbb{N}$, thus showing consistency with the discussion in Remark \eqref{remk:bc_0}. Notably, for $\gamma >3$, there exists only one positive eigenvalue $\lambda_0$, with
\begin{equation}
    \lambda_0 = \mu \in \mathbb{R}_+.
\end{equation}

\begin{remark}
    In accordance with the discussion following Theorem \ref{thm:spectru_Laguerre}, we highlight that for $\gamma \in (2,3]$ the eigenfunctions $\{p_n\}_{n\in \mathbb{N}}$ do not belong to the Hilber space $L^2((0,\infty),w_\gamma)$ because
\begin{equation}
p_n^2(x) w_\gamma(x) \underset{x \to +\infty}{\sim} x^{2-\gamma}.
\end{equation}
\end{remark}

\subsection{Integrability Properties of Eigenfunctions of the CEV Operator}

An eigenfunction to the CEV operator defines a probability density if it is non-negative and integrable. Under the transformation $X\to Y$, integrability  is preserved as follows: For $\gamma<2$, we have that
\begin{align}
    \int_0^\infty p_X(x)\, dx &= \int_0^\infty p_Y(y) \, dy <\infty, \quad \gamma <2,\\
    \int_0^\infty p_X(x)\, dx & =\int_0^\infty p_Y(y) \, dy = 
\int_0^\infty e^{-y} g(y) \, dy <\infty, \quad \gamma >2.
\end{align}
In both regimes, the integrand on the right-hand side corresponds to an eigenfunction of the generalized Laguerre operator \eqref{defLa}.\\

\noindent
\underline{$1\le\gamma < 2$ ($1\le a<\infty$)}:\\ The only eigenfunctions of the operator are the generalized Laguerre polynomials. Since these are non-null polynomials, they are clearly not integrable on $(0,+\infty)$. Consequently, the CEV operator possesses no integrable eigenfunctions in this parametric domain.\\

\noindent
\underline{$0<\gamma < 1$ ($1/2<a<1$)}:\\
For $\Lambda \neq -n$, $n\in \mathbb{N}$ (or equivalently, when the eigenfunction is associated with a self-adjoint extension $\theta \in \mathbb{R}$), the exponential growth of $\Phi$ at infinity excludes it as a potential eigenfunction, thus implying $C_1=0$ in \eqref{sol_kummer}. The asymptotic behaviors described in \eqref{eq:as_phi_Psi_inf} and \eqref{eq:as_Psi_0} then restrict the eigenvalue to $\Lambda >1$ to ensure the integrability of $\Psi$ on $(0,+\infty)$. Under these conditions, the boundary condition \eqref{eq:self_adjoint_condition} takes the form 
\begin{equation}\label{eq:theta_gamma}
    \frac{\Gamma(-a)}{\Gamma(\Lambda-a)} + \theta \frac{\Gamma(1+a)}{\Gamma(\Lambda)} =0, 
\end{equation}
which, for any $\Lambda >1$, admits a unique solution $\theta= \theta(\Lambda)$, determining a unique self-adjoint extension. It can be readily verified from \eqref{eq:spec_theta_R} that $\lambda$ belongs to the spectrum of the extension associated with $\theta= \theta(\Lambda)$. By \eqref{eq:sigma_gamma<2}, the corresponding eigenvalue of $\mathcal{L}_\gamma$ is 
\begin{equation}
    \lambda = \mu(2-\gamma)(\Lambda -1) >0,
\end{equation}
associated to the eigenfunction
\begin{equation}
    p^c_\lambda(x) =  c|\nu| (2-\gamma) x^{1-\gamma} \Psi\bigg(\Lambda, 1+ \frac{1}{2-\gamma}, |\nu| x^{2-\gamma}\bigg),
\end{equation}
$c\in \mathbb{R}/\{0\}$. The parameter range ensures the non-negativity of $\Psi$, allowing $p^c_\lambda$ to define a valid density for any $c>0$. The asymptotic behavior of the eigenfunction is characterized by:
\begin{equation}
p^c_\lambda(x) \sim K_1x^{-\gamma} \text{ as } x \to 0, \quad p^c_\lambda(x) \sim K_2x^{(2-\gamma)(1-\Lambda) - 1} \text{ as } x \to \infty,
\end{equation}
where $(K_1,K_2)$ are positive constants. Hence, for any $\lambda >0$, there exists a unique self-adjoint extension $\theta(\Lambda)\in \mathbb{R}$ such that $\lambda$ belongs to its spectrum and the associated eigenfunction $p^c_\lambda$ defines a valid density on $(0,+\infty)$. 
\begin{remark}
    We emphasize that a defining property of the eigenfunction $p^c_\lambda$ is the accumulation of mass at the origin $x=0$. This behavior is consistent with the boundary classification for $\gamma<2$, where the $0$-boundary is not only attainable but serves as an absorbing state for the underlying process $X$.
\end{remark}

Turning to the extension $\theta=\infty$, which requires separate treatment, the same reasoning for $\gamma \in [1,2)$ holds. This excludes the existence of integrable eigenfunctions for this specific self-adjoint extension.\\

\noindent
\underline{$\gamma >2$ ($a<0$)}:\\
We first consider the extension $\theta = \infty$. Given $\nu>0$ in this regime, the exponential term in \eqref{eigenfunc_ge_2} ensures integrability at the origin for all $n\in \mathbb{N}$. Conversely, as $x\to \infty$, the eigenfunctions exhibit the asymptotic decay  $ p_n(x) \sim x^{1-\gamma}$, which is sufficient to guarantee integrability for $\gamma>2$. Furthermore, according to the spectral relation \eqref{specLgam_ge_2}, the positivity of the eigenvalue $\lambda$ is equivalent to the condition $n \le \frac{1}{\gamma-2}$. \\

For the self-adjoint extensions $\theta \in \mathbb{R}$, the requirement $\lambda \geq 0$ $C_1=0$ in \eqref{sol_kummer}, due to the asymptotic growth of $\Phi$ at infinity \eqref{eq:as_phi_Psi_inf}.  Consequently, for each eigenvalue 
\begin{equation}
\lambda = \mu(\gamma-2)\left(\Lambda +\frac{1}{\gamma-2}\right)\geq 0,
\end{equation}
or equivalently $\Lambda \geq -\frac{1}{\gamma-2}$, there exists a unique self-adjoint extension $\theta = \theta(\Lambda)$ defined by \eqref{eq:theta_gamma}, for which $\lambda$ belongs to its spectrum. The associated eigenfunction, 
\begin{equation}
    p^c_\lambda(x) =  c|\nu| (2-\gamma) x^{1-\gamma} e^{-|\nu|x^{2-\gamma}} \Psi\bigg(\Lambda, 1+ \frac{1}{2-\gamma}, |\nu| x^{2-\gamma}\bigg),
\end{equation}
defines a density for any $c<0$ and possesses infinitely many moments. Furthermore, we observe that
\begin{equation}
    \lim_{x\to 0}p^c_\lambda(x) = \lim_{x\to +\infty}p^c_\lambda(x) =0,
\end{equation}
in accordance with Remark \ref{remk:bc_0}. On the contrary, for $\lambda <0$, it is not needed to set $C_1=0$. However, each associated eigenfunction $p_\lambda$ would still satisfy $\lim_{x\to 0}p_\lambda(x) =0$.

\section{Arbitrage Theory of the CEV Model}\label{sec:arbitrage_CEV}
\subsection{Probabilistic arbitrage theory}
In this section, we recall the classical arbitrage theory of the CEV model. The analysis was carried out in a high level of details in \cite{Delbean}, while \cite{MijatovicUrusov2012} discussed how some different notions of arbitrage may actually hold for $\gamma > 2$, due to the loss of the martingale property in the discounted asset price (\cite{EmanuelMacBeth1982,CarrChernyUrusov2007_Preprint}). We again distinguish various regimes depending on the parameter $\gamma$. The case $\gamma<2$ was treated in \cite{Delbean}, where a relation to a squared Bessel process with dimension \eqref{defdelta} was exploited. For the CEV model with $\gamma>2$, we refer to \cite{MijatovicUrusov2012}, and in particular Example 2.8. For $\gamma=2$, the CEV model reduces to the well-known Black-Scholes model. In the original paper \cite{black1973pricing}, the authors derived their model from no-arbitrage assumptions. \\

We consider the market dynamics defined in \eqref{generic_market}, where the spot price $X$ evolves according to the CEV process specified in \eqref{CEV}. The classical arbitrage theory of the CEV model is closely linked to Bessel processes, whose dimension we denote as 
\begin{equation}\label{defdelta}
    \delta = 2\frac{1-\gamma}{2-\gamma}. 
\end{equation}
Let $X_t^\delta$ be the Bessel process of dimension $\delta$, satisfying the dynamics
\begin{equation}
    \mathrm{d}X_t^\delta = \delta \mathrm{d}t + 2 \sqrt{X^\delta_t}\mathrm{d}W_t, \quad X^\delta_0 = x^\frac{2}{2-\delta},
\end{equation}
and define $\zeta$ the first passage time of $X^\delta$ in $0$, i.e., 
\begin{equation}
\tau=\inf \left\{t>0 ; X_t^{(\delta)}=0\right\},
\end{equation}
Following the results of [\cite{Delbean}, Section 2], the first passage time for the process $X$ is given by 
\begin{equation}
\begin{aligned}
&\zeta=\frac{1}{\mu(2-\gamma)} \log \left(\frac{\sigma^2(1-\gamma/2)}{\sigma^2(1-\gamma/2)-2 \mu \zeta}\right) \text { on the set }\left\{\zeta<\frac{\sigma^2(2-\gamma)}{4 \mu}\right\}
\end{aligned}
\end{equation}
\\
 and by
 \begin{equation}
 \zeta=+\infty \text { on the set }\left\{\zeta \geq \frac{\sigma^2(2-\gamma)}{4 \mu}\right\}.
 \end{equation}
Notably the boundary $0$ is never attainable for $\gamma\geq2$ (see [\cite{Delbean}, Remark 2.1]). On the contrary, for $\gamma <2 $ the boundary $X=0$ is attainable and absorbing, so that $X$ hits zero with positive probability for any $t>0$. $X$ remains equal to zero as soon as it touches the boundary. Following Proposition \ref{prop_suff_cond}, the change of variables is defined via
\begin{align*}
   Z_t &= \exp\left(-\frac{\mu -r}{\sigma}\int_0^t X_s^{1- \gamma/2} \,\mathrm{d}W_s - \frac{1}{2}\bigg(\frac{\mu -r}{\sigma}\bigg)^2\int_0^t X_s^{2-\gamma}\mathrm{d}s\right),\\
   &= \exp\left(-\frac{\mu -r}{\sigma}\int_0^{t \wedge \zeta} X_s^{1- \gamma/2} \,\mathrm{d}W_s - \frac{1}{2}\bigg(\frac{\mu -r}{\sigma}\bigg)^2\int_0^{t \wedge \zeta} X_s^{2-\gamma}\mathrm{d}s\right),\quad t \in [0,T] 
\end{align*}
For any $T <+\infty$, the process $\{Z_t\}_{t\in [0,T]}$ is a $(\mathbb{F},\mathbb{P})$-martingale ([\cite{Delbean}, Theorem 2.3), and $Z_T$ acts as density for the unique change of measure such that 
\begin{equation}
Z_T = \frac{\mathrm{d}\mathbb{Q}}{\mathrm{d}\mathbb{P}} \quad \text{on } \mathcal{F}_T.
\end{equation}
We highlight that some care is needed for $T=+\infty$: \cite{MijatovicUrusov2012} point out that $Z$ may fail to be a martingale on the time horizon $[0,+\infty]$ (which we however do not consider). Under $\mathbb{Q}$, the dynamics of $X$ and $U$ are
\begin{equation}\label{discounted_dyn}
\mathrm{d} X_t  =r X_t \mathrm{d} t+\sigma X_t^\frac{\gamma}{2} \mathrm{d} \tilde{W}_t, \; X_0=x, \quad
\mathrm{d} U_t  =\sigma e^{-r(1-\frac{\gamma}{2})t} U_t^\frac{\gamma}{2} \mathrm{d} \tilde{W}_t, \; U_0=x,
\end{equation}
where $\tilde{W}$ is an $(\mathbb{F},\mathbb{Q})$-Brownian motion defined by
\begin{equation}
\tilde{W}_t = W_t + \int_0^t\frac{\mu-r}{\sigma} X^{1-\frac{\gamma}{2}}\mathrm{d}s, \quad t\in [0,T].
\end{equation}
There exists an ELMM $\mathbb{Q}$ and, by Theorem \ref{FTAP}, the (NFLVR) condition holds. We proceed distinguishing different cases. 

\begin{itemize}
\item[$(i)$] \underline{\textbf{$\gamma=2$.}}\\
The CEV model reduces to classical Geometric Brownian Motion The change of variable previously defined is the actually the "classical" one employed in the Black-Scholes model.\\

\item[$(ii)$] \underline{\textbf{$\gamma<2$, or $\delta\in(-\infty,1)$.}}\\
While (NFLVR) holds, constraining $X$ to remain strictly positive generates arbitrage opportunities. Notably, a new probability measure $\mathbb{M}$ can be introduced, which is equivalent to the conditional measure $\mathbb{P}[\cdot\mid X_T>0]$ on $\mathcal{F}_T$. Defining the auxiliary process $\hat{X}$, whose $\mathbb{Q}$-dynamics are
\begin{equation}\label{mod_dyn}
\mathrm{d} \hat{X}_t =[r \hat{X}_t +\sigma^2 \hat{X}_t^{\gamma-1}]\mathrm{d} t+\sigma \hat{X}_t^\frac{\gamma}{2} \mathrm{d} \tilde{W}_t, \quad \hat{X}_0=x>0,
\end{equation}
it is shown in [\cite{Delbean}, Lemma 4.1], that
\begin{equation}
\left\{X_t ; 0 \leq t \leq T\right\} \text { under } \mathbb{M} \stackrel{\text { law }}{=}\left\{\hat{X}_t ; 0 \leq t \leq T\right\} \text { under } \mathbb{Q}.
\end{equation}
An arbitrage strategy for $\hat{X}$ then follows from classical results on Bessel processes, notably from [\cite{DelbaenSchachermayer1995}, Theorem 6]. The key point is that $\hat{X}$ can be interpreted as a modified Bessel process with parameter $\hat{\delta}>2$.\\

\item[$(iii)$]\underline{\textbf{$\gamma>2$, or $\delta\in(2,+\infty)$.}}\\
Although the equivalent probability measure $\mathbb{Q}$ is well-defined, the process $U$ is now a strict $(\mathbb{F},\mathbb{Q})$-local martingale. As previously mentioned, this loss of the martingale property has been documented since \cite{EmanuelMacBeth1982}. In particular, $\mathbb{E}^{\mathbb{Q}}[U_T]<+\infty$ for all $T>0$, but the martingale property fails and the expected value is not preserved over time. While the market still satisfies (NFLVR) ($U$ is a local martingale bounded from below) the strict local martingale behavior is commonly interpreted as evidence of a price bubble.\\Since the boundary $0$ is never attainable, the arbitrage construction based on conditioning $X$ to remain strictly positive does not apply in this regime.
\end{itemize}

\subsection{Connection between Arbitrage and Spectral Properties }\label{subsec:arb_spec}

\paragraph{Conditioning and its spectral meaning.}
Let
\begin{equation}
\mathcal{G}_\gamma[f]
=
\mu x \,\partial_x f
+\frac{\sigma^2}{2}x^\gamma \,\partial_{xx}f,
\qquad x>0,
\end{equation}
be the generator of the CEV diffusion, and let $\mathcal L_\gamma$ denote the corresponding forward Fokker--Plank operator.  
For $0<\gamma<2$, the boundary point $x=0$ is accessible. Hence the event of never hitting $0$ is non-trivial, and conditioning the process to remain positive is implemented by a Doob $h$-transform. The relevant function $h$ is the strictly positive solution of
\begin{equation}
\mathcal{G}_\gamma[h]=0,
\qquad
h(0)=0,
\qquad
h(x)>0 \ \text{for }x>0.
\end{equation}
Writing this equation explicitly,
\begin{equation}
\frac{\sigma^2}{2}x^\gamma \partial_{xx} h(x)+\mu x \partial_x h(x)=0.
\end{equation}
A direct integration gives
\begin{equation}
\partial_x h(x)=C \exp\!\bigl(\nu x^{2-\gamma}\bigr),
\qquad
\nu=\frac{2\mu}{(\gamma-2)\sigma^2},
\end{equation}
and therefore
\begin{equation}
h(x)=C_1+C_2\int_0^x \exp\!\bigl(\nu y^{2-\gamma}\bigr)\,dy.
\end{equation}
For $0<\gamma<2$ one has $\nu<0$, and the conditioning to stay positive selects the normalized positive harmonic function vanishing at the origin,
\begin{equation}
h_\uparrow(x)=\int_0^x \exp\!\bigl(\nu y^{2-\gamma}\bigr)\,dy,
\end{equation}
and its generator is given by
\begin{equation}
\mathcal{G}_\gamma^\uparrow[ f]
:=
h_\uparrow^{-1}\mathcal{G}_\gamma[h_\uparrow f]
=
\frac{\sigma^2}{2}x^\gamma \partial_{xx}f
+
\left(
\mu x+\sigma^2 x^\gamma \frac{\partial_x h_\uparrow}{h_\uparrow}
\right)\partial_x f.
\end{equation}
Thus the conditioning adds the positive drift term
\begin{equation}
\sigma^2 x^\gamma \frac{\partial_x h_\uparrow(x)}{h_\uparrow(x)},
\end{equation}
which pushes the process away from the absorbing boundary at $0$. This is precisely the mechanism by which the conditioned model ceases to coincide with the original no-arbitrage dynamics.\\

\begin{proposition}
Let
\begin{equation}
\mathcal{G}_\gamma [f]
=
\mu x \,\partial_x f
+\frac{\sigma^2}{2}x^\gamma \,\partial_{xx} f,
\qquad x>0,
\end{equation}
be the CEV generator. Then,

\begin{enumerate}
\item If $0<\gamma<2$, then there exists a non-constant positive harmonic function
\begin{equation}
h_\uparrow(x)=\int_0^x \exp\!\bigl(\nu y^{2-\gamma}\bigr)\,dy,
\qquad
\nu=\frac{2\mu}{(\gamma-2)\sigma^2}<0,
\end{equation}
satisfying
\begin{equation}
\mathcal{G}_\gamma [h_\uparrow]=0,
\qquad
h_\uparrow(0)=0,
\qquad
h_\uparrow(x)>0 \ \text{for }x>0.
\end{equation}
The Doob transform by $h_\uparrow$ yields the conditioned generator
\begin{equation}
\mathcal{G}_\gamma^\uparrow[ f]
=
\frac{\sigma^2}{2}x^\gamma \partial_{xx}f
+
\left(
\mu x+\sigma^2 x^\gamma \frac{\partial_x h_\uparrow}{h_\uparrow}
\right)\partial_x f.
\end{equation}
Hence the conditioned dynamics differs from the original one by the additional positive drift
\begin{equation}
\sigma^2 x^\gamma \frac{\partial_x h_\uparrow}{h_\uparrow},
\end{equation}
which is the spectral signature of the arbitrage construction.

\item For $0<\gamma<1$, this non-trivial conditioning is also visible in the Fokker--Plank spectral problem through the existence of positive integrable eigenfunctions for $\theta\neq \infty$. 

\item For $1\le \gamma<2$, the same conditioning is not represented by an integrable eigenfunction for the Fokker--Plank, but only by a generalized boundary state of the spectral problem of the generator $\mathcal{G}_\gamma$.

\item If $\gamma>2$, then $x=0$ is unattainable, so conditioning on non-negativity is trivial. Accordingly, there is no non-constant positive harmonic function attached to the boundary point $0$ that could generate a non-trivial Doob transform. Therefore the above arbitrage construction is not available for $\gamma>2$.
\end{enumerate}
\end{proposition}

\begin{remark}
The relevant spectral object is not just any eigenfunction, but a positive one. In Markov theory, positivity is what singles out harmonic functions that can define Doob transforms. So from the full spectral analysis, one can isolate exactly those modes that are candidates for conditioning.\\
This gives a spectral criterion:
A nontrivial conditioning capable of changing the pricing law can arise only if the operator admits a nonconstant positive harmonic or ground-state-type solution compatible with the boundary classification.\\
For $\gamma>2$, the complete spectral problem has no boundary mode at 0 that can serve as a nontrivial h-transform generating a new positive semigroup.\\
The Doob-transform construction is independent of the full spectral decomposition; it requires only a positive harmonic function of the backward generator. However, the full spectral analysis refines this picture by identifying whether that positive state is an admissible forward eigenfunction, only a generalized boundary state, or absent altogether. This yields a spectral classification of the arbitrage mechanism: for $0<\gamma<1$ it is visible in the Fokker--Plank spectrum, for $1\le \gamma<2$ only in the generalized spectrum, and for $\gamma>2$ it is excluded by the absence of a nontrivial positive boundary state at 0.

\end{remark}

\section{Discussion and Further Perspectives}

From the point of view of mathematical finance, these results place the CEV model within the
spectral approach to derivative pricing initiated by eigenfunction-expansion methods for scalar
diffusions \cite{DavydovLinetsky2001,davydov2003pricing,Linetsky2004CEV}, and connect it to
operator-theoretic approaches to long-run pricing and recovery \cite{hansen2009longterm,ross2015recovery}.
In contrast to purely pricing-oriented spectral decompositions, however, the present analysis focuses
on the relation between the spectrum, boundary behavior and arbitrage. This perspective is close in
spirit to the classical arbitrage theory of Delbaen and Schachermayer
\cite{delbaen1994arbitrage,delbaen1998fundamental,delbaen2006mathematics}, but it reformulates part of the no-arbitrage
structure in terms of positive eigenfunctions, harmonic functions, and admissible self-adjoint
realizations of the generator. The link with Bessel processes, already central in the analysis of
CEV arbitrage \cite{DelbaenSchachermayer1995,Delbean,CarrLinetsky2006}, appears here as a
spectral boundary phenomenon.\\

The contribution of the paper is therefore twofold. First, it gives a unified Sturm--Liouville
description of the CEV operator across all elasticity regimes, including the explicit
Laguerre reduction and the corresponding boundary conditions. Second, it identifies the spectral
objects that are relevant for arbitrage. In particular, the positive harmonic function used in the
Doob $h$-transform for $0<\gamma<2$ is interpreted as the operator-theoretic signature of the
conditioning mechanism that removes paths hitting the absorbing boundary. For $0<\gamma<1$,
this mechanism is visible at the level of positive forward eigenfunctions of the Fokker--Plank. For
$1\leq \gamma<2$, it persists only as a generalized boundary state. For
$\gamma>2$, the origin is unattainable and no analogous positive boundary mode exists. The
arbitrage-related phenomenon is instead the strict-local-martingale behavior of the discounted
asset price, corresponding to the bubble regime studied in
\cite{EmanuelMacBeth1982,Cox,jarrow2007complete,jarrow2010incomplete}.\\

The present analysis also clarifies the role of self-adjoint extensions for the CEV model. In a diffusion model with singular
endpoints, specifying the stochastic dynamics is not always independent of specifying admissible
boundary behavior. The choice of extension determines which spectral modes are allowed and hence which semigroup is being considered. In financial terms, this is closely related to the choice of absorbing, reflecting, killing or entrance-type boundary behavior, and therefore to the admissible pricing operator. This observation is consistent with earlier work showing that pricing problems for singular diffusion operators can be interpreted as extension-selection problems for the generator
\cite{Linetsky2004CEV,Linetsky2008Spectral}. The CEV model is particularly transparent
because the Laguerre transformation makes these choices explicit.\\

There are, however, several limitations to the present approach. The analysis is essentially
one-dimensional and relies on the special reducibility of the CEV operator to a generalized Laguerre operator. This yields closed-form eigenfunctions and spectra, but such explicit solvability is not available for general local volatility or stochastic volatility models. Moreover, the spectral analysis is carried out in weighted Hilbert spaces, while pricing and arbitrage questions often naturally live
in spaces of bounded, integrable, or merely measurable payoff functions. Consequently, care is needed when translating spectral statements into probabilistic statements about equivalent martingale
measures, strict local martingales and admissible trading strategies. Another limitation is that the
discussion of arbitrage remains structural: the paper identifies the spectral mechanisms behind
known arbitrage and bubble phenomena, but it does not construct a complete trading-strategy-level
classification for all choices of boundary condition. Finally, the Black--Scholes case $\gamma=2$ is singular from the present point of view and has to be treated separately by logarithmic coordinates rather than by the Laguerre reduction.\\

A natural continuation is to extend the method to more complex Markovian systems. For
multi-dimensional diffusion processes with state variable $X_t\in D\subseteq\mathbb{R}^d$, the Fokker--Plank operator takes the form
\[
    L[p] = -\sum_{i=1}^d \partial_i(b_i p)
    + \frac12 \sum_{i,j=1}^d \partial_{ij}(a_{ij}p),
\]
where $b$ is the drift vector and $a=\sigma\sigma^\top$ is the diffusion matrix. In this setting,
the one-dimensional Sturm--Liouville classification has to be replaced by the spectral theory of
degenerate elliptic operators on domains with boundary. Boundary accessibility is then no longer
described by two endpoints, but by the geometry of the boundary faces, corners and lower-dimensional
strata of $D$. Nevertheless, the same general principle should persist: arbitrage-relevant changes
of measure are encoded by positive eigenfunctions or harmonic functions of the generator,
and conditioning on survival or non-attainment of a boundary should correspond to Doob transforms
by such functions. The challenge is to characterize the admissible positive spectrum of the generator
under financially meaningful boundary conditions.\\

This perspective is particularly relevant for stochastic volatility models. In the Heston model
\cite{heston1993closed}, for example, the state variables are the asset price and its variance, and
the variance process has a CIR-type boundary at zero. Since CIR and CEV dynamics are both closely
related to squared Bessel processes, one expects boundary classification and positive harmonic
functions to remain central. A two-dimensional spectral approach could analyze the joint generator,
possibly after Fourier transformation in the logarithmic price variable, reducing the problem to a family of one-dimensional operators in the variance coordinate. This is precisely the type of structure for which explicit or semi-explicit spectral methods may remain feasible. Similar ideas could also be applied to multi-factor affine models, local-stochastic volatility models, and interacting systems of assets whose volatility coefficients degenerate at parts of the state-space boundary.\\

In summary, the CEV model provides a useful test case in which the links between spectral theory,
boundary behavior and arbitrage can be made fully explicit. The present results suggest that, more
generally, arbitrage and bubble phenomena in diffusion models may be studied through the positive
spectrum of the pricing generator, the admissible boundary conditions of the associated forward
operator, and the Doob transforms generated by positive harmonic states. Extending this framework
to multi-dimensional and stochastic volatility models remains a promising direction for future work.

\medskip

\noindent
\textbf{Acknowledgment.}\\
The authors would like to thank Josef Teichmann, Walter Schachermayer and Patrick Cheridito for useful comments and suggestions.

\bibliographystyle{abbrv}
\bibliography{Biblio}

\newpage

\appendix

\section{General Sturm--Liouville Theory on Weighted Hilbert Spaces}\label{app:generic_SL}

Let us recall some general properties of Sturm--Liouville operators \cite{zettl2005sturm}. A Sturm--Liouville operator is a second-order linear ordinary differential operator written in normal form as 
\begin{equation}
    \mathcal{L}(p) = \frac{1}{w}\left[\frac{\partial}{\partial x}\left(Q_2 \frac{\partial p}{\partial x}\right)+Q_0p\right],
\end{equation}
where the functions are defined on some possibly infinite interval $I=(a,b)\subseteq \mathbb{R}$ together with some boundary conditions \footnote{Classically, the coefficient functions appearing in the definition of $\mathcal{L}$ are denoted as $q=Q_0$ and $p=Q_2$. Since, however, $p$ is reserved for the probability density function appearing in the Fokker--Planck equation, we changed the notation to avoid confusion. We also emphasize that the sign of the leading-order coefficient $Q_2$ is positive as opposed to the negative sign-convection frequently employed in the literature.} Throughout, we make the minimal assumptions that
\begin{equation}
    \begin{split}
        Q_2^{-1}\ &\in L_{\rm loc}^1(I),\quad Q_0>0,\\
        Q_0 & \in L_{\rm loc}^1(I),\\
        w & \in L_{\rm loc}^1(I),\quad w>0,
    \end{split}
\end{equation}
where $L_{\rm loc}^1(I)$ denotes the space of locally integrable functions on the interval $I$.\\
 
If the interval $I$ is finite, the functions satisfy $w\in C^0(I)$ with $w>0$, $Q_0\in C^0(I)$ as well as $Q_2\in C^1(I)$ with $Q_2>0$ and the boundary conditions are of mixed Dirichlet--Neumann type, the SL problem is called \emph{regular}. If it is not regular, it is called \emph{singular}.\\

The operator $\mathcal{L}$ is naturally defined on the Hilbert space of $w$-weighted square integrable functions over $I$. We denote the $w$-weighted inner products as 
\begin{equation}
\begin{split}
        \langle p,q\rangle_w  = \int_a ^b p(x)q(x)w(x)\, dx,
\end{split}
\end{equation}
with norm 
\begin{equation}
     \|p\|_w^2 = \langle p,p \rangle_w
\end{equation}
and its corresponding Hilbert space as 
\begin{equation}
\begin{split}
    L^2(I, w)  = \{p: I \to \mathbb{R}: \|p\|_w<\infty \}.\\
\end{split}
\end{equation}
The maximal domain of definition of definition of $\mathcal{L}$ in $L^2_w(I)$ is then given by
\begin{equation}
    \mathcal{D}(\mathcal{L}) = \{f\in L^2(I, w): f, Q_2f\in AC_{\rm loc}(I),\mathcal{L}[f]\in L^2_w(I)\},
\end{equation}
where $AC_{\rm loc}(I)$ is the space of locally absolutely continuous functions over $I$. \\

We denote the \emph{modified Wronskian} of two functions $f$ and $g$ at the point $x$ as 
\begin{equation}
    W_x(f,g) = Q_2(x)[f(x)g'(x)-f'(x)g(x)]. 
\end{equation}

The SL operator $\mathcal{L}$ is called \emph{limit circle} (l.c.) at the point $a$ is there exists a function $g\in \mathcal{D}(\mathcal{L})$ such that $W_a(g,f)\neq 0$ for at least one $f\in \mathcal{D}(\mathcal{L})$. Otherwise, $\mathcal{L}$ is called \emph{limit point} (l.p.) at $a$. It follows that $\mathcal{L}$ is limit point if and only if $W_a(f,g)=0$ for all $f,g\in\mathcal{D}(\mathcal{L})$. In the following, when the operator $\mathcal{L}$ is fixed, we will call the boundary points l.c. or l.p. respectively.\\

The following theorem guarantees the self-adjointness of $\mathcal{L}$ for suitable boundary conditions. \\

\begin{theorem}\label{thmselfadjoint}
If $a$ is l.c., let $g_a$ be such that $W_a(g_a,f)\neq 0$ for at least one $f$. Similarly, let $g_b$ an analogous function if $b$ is l.c.. The operator $\mathcal{L}$ defined on the domain 
\begin{equation}
    \mathcal{D}_{a,b} = \{ f\in L^2(I, w): W_a(g_a,f) =0 \quad \text{if a is l.c.},\quad W_b(g_b,f) =0  \quad \text{if b is l.c.} \}
\end{equation}
is self-adjoint and hence, its spectrum is real. 
\end{theorem}

\begin{remark}
The distinction between limit circle and limit point might appear technical at first glance. It only characterizes, however, if the boundary behavior of $Q_2$ is such that we need to specify additional boundary conditions to make $\mathcal{L}$ self-adjoint, or not. This can be seen immediately by integration by parts 
    \begin{equation}
    \begin{split}
        \langle \mathcal{L} f,g\rangle_w &  = \int_a^b \left(\frac{1}{w} \frac{\partial }{\partial x} \left[Q_2\frac{\partial f}{\partial x}\right] + Q_0f\right) g w \, dx\\
        & = \left[Q_2\frac{\partial f}{\partial x} g\right]_a^b + \int_{a}^b Q_0 fg w -Q_2 \frac{\partial f}{\partial x}\frac{\partial g}{\partial x}\, dx \\ 
        & = \left[Q_2\frac{\partial f}{\partial x} g\right]_a^b - \left[Q_2\frac{\partial g}{\partial x} f\right]_a^b+ \int_{a}^b Q_0 fg w + f \frac{\partial }{\partial x}\left[Q_2\frac{\partial g}{\partial x}\right]\, dx\\
        & = W_{a}(f,g)+W_b(f,g)+ \langle g,\mathcal{L} g \rangle. 
                \end{split}
    \end{equation}
    
\end{remark}
We conclude this section with a lemma.
 
\begin{lemma}\label{LemmaNF}
    A Sturm--Liouville operator of the form
    \begin{equation}
        \mathcal{L}(p) = l_0p+l_1\frac{\partial p}{\partial x} + l_2 \frac{\partial^2p}{ \partial x^2},
    \end{equation}
with $l_2>0$ can be brought to normal form 
\begin{equation}\label{SLnormal}
    \mathcal{L}(p) = \frac{1}{w}\left[\frac{\partial}{\partial x}\left(Q_2 \frac{\partial p}{\partial x}\right)+Q_0p\right],
\end{equation}
for the weight function
\begin{equation}\label{defw}
    w(x) = \frac{1}{l_2(x)}\exp\left(\int_{x_0}^x \frac{l_1(\xi)}{l_2(\xi)}\, d\xi\right),
\end{equation}
and the coefficient functions 
\begin{equation}\label{defQ}
    Q_2 = wl_2,\quad Q_0 = w l_0
\end{equation}
\end{lemma}
The proof is standard and we include it here for completeness. 
\begin{proof}
Expanding the normal form \eqref{SLnormal}, we find that the coefficient functions are related as
    \begin{equation}\label{rellQ}
        l_2 = \frac{Q_2}{w},\quad l_1 = \frac{\partial_x Q_2}{w},\quad l_0 = \frac{Q_0}{w}. 
    \end{equation}
    Taking an $x$-derivative of $l_2$ and comparing the expressions for $\partial_x Q_2$, we find that the weight function satisfies the ODE
    \begin{equation}
        \partial_x w = \frac{l_1-\partial_x l_2}{l_2} w,
    \end{equation}
    which can be readily integrated to give expression \eqref{defw}. The expressions \eqref{defQ} then follow immediately from \eqref{rellQ}. 
\end{proof}

\section{Complete Spectral Analysis of the Generalized Laguerre Operator}\label{sec:spectrum_Laguerre}
We introduce the \emph{generalized Laguerre operator}, 
\begin{equation}\label{defLa}
    \mathbf{L}_a [p] :=   yp'' + (a+1-y)p',\quad a\in\mathbb{R},
\end{equation}
where the prime denotes a derivative with respect to $y$. The operator $\mathbf{L}_a$ is itself a Sturm--Liouville operator with weight 
\begin{equation}
    w(y) = y^a e^{-y},
\end{equation}
and Sturm--Liouville normal form coefficients
\begin{equation}
\begin{split}
        Q_0(y) & = -y^1e^y\\
        Q_2(y) & = y^{a+1}e^y. 
\end{split}
\end{equation}

The following results are derived in \cite{derkach1998extensions} and allow a complete spectral characterization of the generalized Laguerre operator \eqref{defLa}. Before we can formulate the main theorem on self-adjoint extensions of $\mathbf{L}_a$ with general boundary conditions, we introduce some notions from operator theory. 

\begin{definition}[Indefinite Inner Products: Krein and Pontrayagin Spaces]
Let $\mathcal{K}$ be a complex vector space. 
A Hermitian sesquilinear form $\langle\cdot,\cdot\rangle : \mathcal{K}\times\mathcal{K}\to\mathbb{C}$ 
is called an \emph{indefinite inner product} if
\begin{equation}
[x,y]=\overline{[y,x]}, \qquad x,y\in\mathcal{K}.
\end{equation}
The pair $(\mathcal{K},\langle\cdot,\cdot\rangle)$ is called an \emph{indefinite inner product space}.\\
A \emph{Krein space} is an indefinite inner product space $(\mathcal{K},\langle\cdot,\cdot\rangle)$
for which there exists a decomposition
\begin{equation}
\mathcal{K}=\mathcal{K}_+ \,\oplus, \mathcal{K}_-
\end{equation}
such that
\begin{enumerate}
\item $(\mathcal{K}_+,\langle\cdot,\cdot\rangle)$ and $(\mathcal{K}_-,-\langle\cdot,\cdot\rangle)$ are Hilbert spaces,
\item $\mathcal{K}_+$ and $\mathcal{K}_-$ are orthogonal with respect to $\langle\cdot,\cdot\rangle$.
\end{enumerate}
Such a decomposition is called a \emph{fundamental decomposition}.\\
A \emph{Pontrayagin space} is a Krein space for which $\dim \mathcal{K}_-=\kappa<\infty$. The number $\kappa$ is called \emph{index}. 
\end{definition}
Following \cite{derkach1998extensions}, for any $a\in\mathbb{R}$, we introduce the indefinite inner product
\begin{equation}\label{definnerproda}
        (f,g)_a
=
\int_0^\infty x^a e^{-x} f(x)\overline{g(x)}\,dx
-
\sum_{j=0}^{n-1}
\frac{f^{(j)}(0)\overline{g^{(j)}(0)}}{j!\,\Gamma(a+j+1)},
    \end{equation}
where $\Gamma$ denotes again the Gamma function. Define the boundary operators 
\begin{equation}
     \mathcal{B}_0 f := \begin{cases}
        \lim_{x\to 0} x^{a+1}f'(x),\quad & a \in (-1,1),\\
        \lim_{x\to 0}\frac{(-1)^n\Gamma(a+1)}{\Gamma(a+n+1)}x^{a+n+1}f^{(n+1)}(x),\quad & a \in (-n-1,-n), \; n\in \mathbb{N} /\{ 0 \},
    \end{cases}
\end{equation}
and 
\begin{equation}
      \mathcal{B}_1 f := \begin{cases}
        \lim_{x\to 0}f(x) + \frac{x}{a}f'(x),\quad & a \in (0,1),\\
        \lim_{x\to 0}f(x),\quad & a \in (-\infty,0),
    \end{cases}
\end{equation}
and consider the Laguerre operator $\mathbf{L}_a$ together with the one-parameter family of boundary conditions
\begin{equation}\label{eq:self_adjoint_condition}
     \mathcal{B}_1 f - \theta  \mathcal{B}_0 f =0, \quad \theta \in \overline{\mathbb{R}}.
\end{equation}
In case $\theta = \infty$, the boundary condition \eqref{eq:self_adjoint_condition} takes the form
\begin{equation}
    \mathcal{B}_0 f =0.
\end{equation}
We introduce the Weil function 
\begin{equation}\label{weil_function}
    m_a(\Lambda) = -\frac{\Gamma(\Lambda) \Gamma(-a)}{\Gamma(\Lambda -a)\Gamma(1+a)}.
\end{equation}

\begin{theorem}[Spectrum of the Laguerre operator for arbitrary $a$]\label{thm:spectru_Laguerre}
Let $a\in\mathbb{R}$ and consider the generalized Laguerre operator \eqref{defLa} and let $\mathcal H(a)$ denote the completion of the space of polynomials with respect to the indefinite inner product \eqref{definnerproda}. The spectrum of $\mathbf{L}_a$ can be characterized as follows: 
\begin{enumerate}
\item[(i)] \underline{The case $a\geq 1$}: The minimal operator generated by $\mathbf{L}_a$ in the Hilbert space $L^2(\mathbb{R}_+,x^a e^{-x}dx)$ is symmetric. In the case $a\geq 1$, both $\infty$ and $0$ are limit point and the operator $\mathbf{L}_a$ is self-adjoint on $L^2_w(0,\infty)$. Its spectrum is discrete and is purely discrete and given explicitly as 
    \begin{equation}\label{eq:spec_Laguerre_theta_infty}
        \sigma(\mathbf{L}_a) = \{-n\}_{ n\in\mathbb{N}}.
    \end{equation}
The corresponding eigenfunctions are the \emph{generalized Laguerre polynomials}, defined by the Rodrigues formula
\begin{equation}\label{defLaguerrepoly}
    L^{a}_{n}(x) = \frac{x^{-a}e^x}{n!}\left(\frac{d}{dx}\right)^{n}(e^{-x}x^{n+a}). 
\end{equation}
\item[(ii)] \underline{The case $|a|<1$}: The operator $\mathbf{L}_a$ has deficiency indices $(1,1)$ and all self-adjoint extensions $\mathbf{L}_{a,\theta}$ are given by the boundary conditions
\begin{equation}
 \mathcal{B}_1 y=\theta\, \mathcal{B}_0 y, \qquad \theta\in\overline{\mathbb{R}}.
\end{equation}
The spectrum of the self-adjoint extension obtained for fixed $\theta \in \mathbb{R}$ consists of simple eigenvalues $\Lambda$, given by the zeros of 
\begin{equation}\label{eq:spec_theta_R}
    m_a(\Lambda)-\theta.
\end{equation}
If instead $\theta = \infty$, it coincides with \eqref{eq:spec_Laguerre_theta_infty}.

\item[(iii)] \underline{The case $-n-1<a<-n$ for some $n\in\mathbb N$}: The completion $\mathcal H(a)$ of polynomials with respect to the indefinite inner product \eqref{definnerproda} is a Pontryagin space with index
\begin{equation}
\kappa=\left\lfloor\frac{n+1}{2}\right\rfloor.
\end{equation}
The minimal Laguerre operator in $\mathcal H(a)$ is a Hermitian operator (not densely defined) with deficiency indices $(1,1)$. Its adjoint admits the boundary triple $(\mathbb C,\Gamma_0,\Gamma_1)$ and
all self-adjoint extensions are given by
\begin{equation}
 \mathcal{B}_1 y=\theta\, \mathcal{B}_0 y, \qquad \theta\in\overline{\mathbb{R}}. 
\end{equation}
In analogy to the Hermitian Hilbert space case, the extension defined by $\theta =\mathbb{R}$ has simple spectrum given by the zeros of \eqref{eq:spec_theta_R}, while the one defined by $ \mathcal{B}_0 y=0$ ($\theta=\infty$) has simple spectrum
\begin{equation}
\sigma(\mathbf{L}_{a,\infty})=\{-n\}_{n\in\mathbb{N}},
\end{equation}
and the corresponding eigenfunctions are again the Laguerre polynomials
$L_n^{a}(x)$ as defined in \eqref{defLaguerrepoly}, which form a complete system in $\mathcal H(a)$.
\end{enumerate}
\end{theorem}

The positive integers are denoted as $\mathbb{N}$ and are assumed to include zero. 

\subsection{Eigenvalues of the Laguerre Operator}
The eigenvalue equation associated to the generalized Laguerre operator \eqref{defLa},
\begin{equation}
    \mathbf{L}_ap = \Lambda p,
\end{equation}
is called \emph{Kummer equation} and for $a\neq -n, \Lambda \neq m,\; (n,m)\in \mathbb{N}^2$, its general solution is given by  
\begin{equation}\label{sol_kummer}
    p_\Lambda(y) = C_1 \Phi(\Lambda,1+a,y) + C_2 \Psi(\Lambda,1+a,y), \quad 
\end{equation}
where $\Phi$ and $\Psi$ are certain degenerate hypergeometric functions, see \cite{abramowitz1948handbook, derkach1998extensions}, and $(C_1,C_2)\in\mathbb{R}^2$. Functions of the form \eqref{sol_kummer} are the only possible eigenfunctions of the generalized Laguerre operator. 
They satisfy the asymptotics
\begin{equation}\label{eq:as_phi_Psi_inf} 
    \Phi(\Lambda,1+a,y) = \frac{\Gamma(1+a)}{\Gamma(\Lambda)}e^y y^{\Lambda-a-1}\left(1+\mathcal{O}\left(\frac{1}{y}\right)\right),\quad \Psi(\Lambda,1+a,y) = y^{-\Lambda}\left(1+\mathcal{O}\left(\frac{1}{y}\right)\right),
\end{equation}
for $y\to \infty$, as well as the asymptotics
\begin{equation}
    \Phi(\Lambda,1+a,0) = 1,
\end{equation}
\begin{equation}\label{eq:as_Psi_0}
    \Psi(\Lambda,1+a,y) \sim \begin{cases}
        y^{-a}\Gamma(a)/\Gamma(\Lambda),\quad & a\in(0,\infty),\\
        -\log(y)/\Gamma(\Lambda),\quad & a = 0,\\
        \Gamma(-a)/\Gamma(\Lambda-a),\quad & a \in (-1,0),\\
         1/\Gamma(\Lambda),\quad & a =-1, \\
         \Gamma(-a)/\Gamma(\Lambda-a),\quad & a \in (-\infty,-1).

    \end{cases}
\end{equation}
for $y\to 0$, see \cite{abramowitz1948handbook}. Furthermore, it follows again from \cite{abramowitz1948handbook,derkach1998extensions} that $\Psi^\prime (A,B,y) = -A \Psi(A+1,B+1,y)$, for $(A,B) \in \mathbb{R}^2$. Hence, for $y\to +\infty$,
\begin{equation}
    \Psi^\prime(\Lambda,1+a,x) \sim -\Lambda x^{-(1+\Lambda)},
\end{equation}
while, for $y\to 0^+$,
\begin{equation}\label{eq:as_psi_0_prime}
\Psi^{\prime}(\Lambda, 1+a, y)  \sim \begin{cases}
        -\Lambda \frac{\Gamma(1+a)}{\Gamma(1+\Lambda)} y^{-(1+a)},\quad & a\in(-1,\infty),\\
        \frac{\Lambda}{\Gamma(\Lambda+1)}\log(y), \quad & a=-1,\\
        - \Lambda \frac{\Gamma(-1-a)}{\Gamma(\Lambda - a)} ,\quad & a \in (-\infty,-1).
        \end{cases}
\end{equation}

\section{Explicit calculations on the transformation between $p_X$ and $p_Y$}

\subsection{Boundary conditions for self-adjoint operator}\label{app:boundary_conditions}

In this section we detail the calculations connecting the boundary condition needed for the l.c. point $0$ in the Laguerre operator (given by \eqref{eq:self_adjoint_condition}) and the ones need for the CEV Fokker--Plank operator. For the latter, we recall that the point $0$ is l.c. for $\gamma \in (1,2)$ (see \eqref{bc_gamma_1_2}) while $\infty$ is a l.c. for $\gamma\in (2,+\infty)$ (see \eqref{bc_gamma>2}). Before proceeding, we recall  the following relationships
\begin{equation}
    \begin{aligned}
        p_{\tilde{Y}}(\tilde{y}) &= p_{Y}(y) \implies p'_Y(y) = |\nu|p_{\tilde{Y}}'(\tilde{y});\\
        p_X(x) &= (2-\gamma)x^{1-\gamma}p_Y(x^{2-\gamma}) \implies\\ 
        p_{Y}'(y) &= \frac{1}{(2-\gamma)^2}\big[x^{2\gamma -2}p_X(x)' + (\gamma-1) p_X(x)x^{2\gamma-2}\big].
    \end{aligned}
\end{equation}

$(i)$ We set ourselves in case $a>0$, or equivalently $\gamma \in (0,2)$. If $a\geq1$, or equivalently $\gamma \leq 1$, no further boundary conditions are needed as both $\tilde{y}=0$ and $\tilde{y}=\infty$ are limit points. In the other case $a\in (0,1)$, or equivalently $\gamma \in (1,2)$, the boundary condition required for ensuring a self-adjoint extension is the same given in Equation \eqref{eq:self_adjoint_condition}. Moving from the variable $\tilde{Y}$ back to $Y$, we obtain
\begin{equation}
    \lim_{y\to 0} p_{Y}(y) + (2-\gamma)y p^{\prime}(y) - \theta |\nu|^{\frac{1}{2-\gamma}}y^{1+\frac{1}{2-\gamma}}p_{Y}'(y)  =0
\end{equation}
Moving now back from $Y$ to $X$, we conclude
\begin{equation}
\begin{aligned}
    &\lim_{x\to 0}  \frac{p_X(x)}{(2-\gamma) x^{1-\gamma}}+ \frac{1}{2-\gamma}\bigg[p'_X(x)x^{\gamma} + (\gamma-1) p_X(x)x^{\gamma-1}\bigg] - \theta |\nu|^{\frac{1}{2-\gamma}} x^{3-\gamma}\left[\frac{p_{X}'(x)}{ (2-\gamma)x^{1-\gamma}}- \frac{1-\gamma}{2-\gamma} \frac{p_X(x)}{x^{2-\gamma}} \right] =\\
    &\lim_{x\to 0}  \frac{\gamma}{\gamma-2 }p_X(x)x^{\gamma-1} + \frac{1}{2-\gamma}p'_X(x)x^{\gamma} - \frac{\theta |\nu|^{\frac{1}{2-\gamma}}}{(2-\gamma)^2} \left[p_{X}'(x)x^{\gamma+1}+ (\gamma-1) p_X(x)x^{\gamma} \right] =\\
    &\lim_{x\to 0}  \frac{\gamma}{\gamma-2 }p_X(x)x^{\gamma-1} + \frac{1}{2-\gamma}p'_X(x)x^{\gamma} - \frac{\theta |\nu|^{\frac{1}{2-\gamma}}}{(2-\gamma)^2}x^{\gamma} \left[p_{X}'(x)x+ (\gamma-1)p_X(x)\right]=0.
\end{aligned}
\end{equation}
 Notably, for the self-adjoint extension $\theta = +\infty$, the following boundary condition has to be imposed
\begin{equation}
\lim_{x\to0}x^\gamma[p'_X(x)x + (\gamma-1)p_X(x)]=0.
\end{equation}

\medskip
$(ii)$ We set ourselves in case $\gamma>2$, or equivalently $a = \frac{1}{2-\gamma} <0$. We recall that in this case $\nu >0$. We are interested in understanding the behavior of the boundary condition required for a self-adjoint extension. We recall that, for the generalized Laguerre operator, in this case $\tilde{y}=\infty$ is a limit point, while $\tilde{y}=0$ is a limit circle. Since in the $\tilde{Y}$ variable it holds that $p_{\tilde{Y}}(\tilde{y}) = e^{-\tilde{y}}f(\tilde{y})$, and clearly $\lim_{\tilde{y}\to0}e^{-\tilde{y}}=1$, condition \eqref{eq:self_adjoint_condition} yields
\begin{equation}\label{bound_cond_y_tilde}
    \lim_{\tilde{y}\to 0} p_{\tilde{Y}}(\tilde{y}) - \theta \tilde{y}^{1+\frac{1}{2-\gamma}}[p_{\tilde{Y}}'(\tilde{y}) + p_{\tilde{Y}}(\tilde{y})] =0
\end{equation}
Moving back to $Y$, one obtains 
\begin{equation}\label{bound_cond_y}
    \lim_{y\to 0} p_{Y}(y) - \theta \nu^{\frac{1}{2-\gamma}}y^{1+\frac{1}{2-\gamma}}\left[p_{Y}'(y) + \nu p_{Y}(y)\right] =0
\end{equation}
 Finally, moving back to the variable $X$,
\begin{equation}\label{bound_cond_x}
\begin{aligned}
    &\lim_{x\to +\infty}  \frac{p_X(x)}{(2-\gamma) x^{1-\gamma}} - \theta \nu^{\frac{1}{2-\gamma}} x^{3-\gamma}\left[\frac{p_{X}'(x)}{(2-\gamma)^2x^{2-2\gamma}}+ \frac{\gamma-1}{(2-\gamma)^2} \frac{ p_X(x)}{x^{3-2\gamma}} + \frac{\nu p_X(x)}{(2-\gamma)x^{1-\gamma}}\right] =\\
    &\lim_{x\to +\infty}  -\frac{p_X(x)x^{\gamma-1}}{\gamma-2 } - \theta \nu^{\frac{1}{2-\gamma}} \left[\frac{p_{X}'(x)x^{1+\gamma}}{ (\gamma-2)^2}+ \frac{\gamma-1}{(\gamma-2)^2} p_X(x)x^\gamma - \frac{2\mu p_X(x)x^2}{(\gamma-2)^2\sigma^2}\right] =\\
    &\lim_{x\to +\infty}  -\frac{p_X(x)x^{\gamma-1}}{\gamma-2 } - \theta \nu^{\frac{1}{2-\gamma}}  \frac{2}{(\gamma-2)^2\sigma^2}x\left[\frac{\sigma^2}{2}p_{X}'(x)x^{\gamma}+ \frac{\sigma^2(\gamma-1)}{2} p_X(x)x^{\gamma-1} - \mu p_X(x)x\right]=0
\end{aligned}
\end{equation}

We shall denote the flux associated to the Fokker-Plank equation by $F(p_X(x))$, and conclude that the boundary condition reads
\begin{equation}
\lim_{x\to +\infty}  -\frac{p_X(x)x^{\gamma-1}}{\gamma-2 } + \theta \frac{\nu^{\frac{1}{2-\gamma}} }{\eta} x\left[F(p_X(x))+\frac{\sigma^2}{2}p_X(x)x^{\gamma-1}\right]=0
\end{equation}
In particular, for $\theta = +\infty$ (case corresponing to Laguerre polynomials), the boundary condition reads
\begin{equation}
\lim_{x\to +\infty}  x\left[F(p_X(x))+\frac{\sigma^2}{2}p_X(x)x^{\gamma-1}\right]=0.
\end{equation}

\subsection{Transforming the Densities of the Fokker--Planck Equation}
\label{checktransform}
In this appendix, We show here that the transformation of densities $p_X$ and $p_Y$ under \eqref{pxpy} correctly relates the Fokker--Planck equations \eqref{FP_eq_CEV} and \eqref{FP_eq_CEV_mod}.\\
First, we calculate the drift term:
\begin{equation}
\begin{split}
       -\partial_x(\mu x p_X) & = -\mu p_X -\mu x \partial_x p_X \\
       & = -\mu (2-\gamma)^2 y^{\frac{1-\gamma}{2-\gamma}} p_Y -\mu (2-\gamma)^2y^{\frac{1-\gamma}{2-\gamma}} y \partial_y p_Y \\
       & =- (2-\gamma)y^{\frac{1-\gamma}{2-\gamma}} \partial_y[(2-\gamma)\mu y p_y].
\end{split}
\end{equation}
Similarly, we calculate the diffusion term:
\begin{equation}
\begin{split}
       \frac{\partial^2}{\partial x^2}\left[ x^\gamma p_X\right] & = \gamma (\gamma-1) x^{\gamma -2}p_X + 2\gamma x^{\gamma-1}\partial_x p_X + x^\gamma \partial_{xx}p_X  \\
       &= \gamma(\gamma -1) (2-\gamma)y^\frac{-1}{2-\gamma}p_Y + 2\gamma(2-\gamma) y^{\frac{\gamma -1 }{2-\gamma}}\partial_x[y^\frac{1-\gamma}{2-\gamma}p_Y] + y^\frac{\gamma}{2-\gamma}\partial_{xx}[y^\frac{1-\gamma}{2-\gamma}p_Y]
\end{split}
\end{equation}
We can now calculate, separately, 
\begin{equation}
\partial_xX[y^\frac{1-\gamma}{2-\gamma}p_Y] = (1-\gamma) y^\frac{-\gamma}{2-\gamma}p_Y + (2-\gamma)y^\frac{2-2\gamma}{1-\gamma}\partial_y p_Y
\end{equation}
and
\begin{equation}
       \partial_{xx}[y^\frac{1-\gamma}{2-\gamma}p_Y]  = -\gamma (1-\gamma) y^\frac{-1-\gamma}{2-\gamma}p_Y + 3(2-\gamma)(1-\gamma)y^\frac{1-2\gamma}{2-\gamma}\partial_yp_Y + (2-\gamma)^2 y^\frac{3-3\gamma}{2-\gamma}\partial_{yy}p_Y. 
\end{equation}
Putting the above together, we obtain
\begin{equation}
\begin{split}
       \frac{\partial^2}{\partial x^2}\left[\frac{\sigma^2}{2} x^\gamma p_X\right] &= 0p_Y + \frac{\sigma^2}{2}(2-\gamma)^2(3-\gamma) y^\frac{1-\gamma}{2-\gamma}\partial_y p_Y + \frac{\sigma^2}{2}(2-\gamma)^3 y^\frac{3-2\gamma}{2-\gamma}\partial_{yy}p_Y  \\
       &= \frac{\sigma^2}{2}(2-\gamma)^2y^\frac{1-\gamma}{2-\gamma}[(3-\gamma)\partial_y p_Y + (2-\gamma) y \partial_{yy}p_Y] \\
       & = (2-\gamma)y^\frac{1-\gamma}{2-\gamma} \partial_{yy}\left[\frac{\sigma^2}{2}(2-\gamma)^2 y p_Y\right] - (2-\gamma)y^\frac{1-\gamma}{2-\gamma} (1-\gamma)(2-\gamma)\partial_y p_Y,
\end{split}
\end{equation}
and arrive at 
\begin{equation}
\partial_t p_Y (2-\gamma)y^\frac{1-\gamma}{2-\gamma} = (2-\gamma)y^\frac{1-\gamma}{2-\gamma}\left\{\partial_y\left[(2-\gamma)\left(\mu y + \frac{\sigma^2}{2}(1-\gamma)\right)\right] + \partial_{yy}\left[\frac{\sigma^2}{2}(2-\gamma)^2 y p_Y\right]\right\},
\end{equation}
which, up to removing the term $(2-\gamma)y^\frac{1-\gamma}{2-\gamma}$, is the Fokker-Plank equation for $Y$.

 \end{document}